\documentclass[11pt,preprintnumbers,reqno]{amsart}
\usepackage{graphicx}
\usepackage{color}
\usepackage{bm}
\usepackage{amssymb}

\newtheorem{remark}{Remark}[section]
\newtheorem{definition}{Definition}[section]
\newtheorem{theorem}{Theorem}
\newtheorem{proposition}{Proposition}[section]
\newtheorem{lemma}{Lemma}[section]

\allowdisplaybreaks
\numberwithin{equation}{section}
\numberwithin{figure}{section}
\usepackage{geometry}
\makeatletter
\renewcommand\subsection{\@startsection{subsection}{2}%
  \z@{-0.8\linespacing\@plus-0.7\linespacing}{0.7\linespacing}%
  {\normalfont\bfseries}}
\renewcommand\subsubsection{\@startsection{subsubsection}{3}%
  \z@{-0.8\linespacing\@plus-0.7\linespacing}{0.7\linespacing}%
  {\normalfont\itshape}}
\makeatother

\begin{document}

\title{Nonlinear discrete Schr\"odinger equations with a point defect}

 \author{Dirk Hennig}
 \email[Email: ]{dirkhennig@uth.gr}
 \address{Department of Mathematics, University of Thessaly, Lamia GR35100, Greece}


\begin{abstract}
\noindent 
We study the $d$-dimensional  discrete nonlinear Schr\"odinger equation with
general power nonlinearity and a delta potential. Our interest lies
in the interplay between two localization mechanisms. On the one hand, the attractive (repulsive) delta potential
acting as a point defect  breaks the translational invariance of the lattice so that a linear staggering (non-staggering) bound state  is formed with negative (positive) energy. On the other hand, focusing nonlinearity may lead to  self-trapping of excitation energy. For focusing nonlinearity we prove the existence of spatially exponentially localized and time-periodic
ground states  and investigate the impact of an attractive respectively repulsive
delta potential on the existence of an excitation threshold, i.e. supercritical $l^2$ norm, for the creation of such  ground states. Explicit expressions for the excitation thresholds are given.
Reciprocally, we discuss the influence of defocusing nonlinearity on the durability of 
the linear bound states and provide upper thresholds of the $l^2$ norm for their preservation. Regarding the asymptotic behavior of the solutions we establish
that for a  $l^2$ norm below the excitation threshold the solutions scatter to a solution of the
linear problem in $l^p$, $2<p\le \infty$.
\end{abstract}

\maketitle

\noindent 
\section{Introduction}

We consider the discrete  nonlinear Schr\"odinger equation (DNLS) on 
the $d$-dimensional infinite  lattice with a delta potential 
\begin{equation}
i\frac{d u_n}{dt}+\kappa (\Delta u)_n+\gamma|u_n|^{2\sigma}u_n+V_0 \delta_{n,0}u_n=0,\,\,\,n \in \mathbb{Z}^d,\label{eq:system1}
\end{equation}
where $u_n\in \mathbb{C}$, $V_0 \in \mathbb{R}$ and $\delta_{m,n}$ denotes the Kronecker delta, i.e.  $\delta_{m,n}=1$ if $m=n$ and $\delta_{m,n}=0$ if $m\neq n$.  That is, the last term on the left-hand side of (\ref{eq:system1}) represents a point defect on the lattice $\mathbb{Z}^d$.
We will refer to (\ref{eq:system1}) also as $\delta$DNLS. The operator $(\Delta u)_n$ is the $d$-dimensional discrete Laplacian
\begin{eqnarray*}
(\Delta u)_{n}=\sum_{j} (u_{n+j}-2 u_n+u_{n-j}),
\end{eqnarray*}
where $j$ are the $d$ unit vectors belonging to the $d$ axes of $\mathbb{Z}^d$ 
and the value of $\kappa>0$ regulates the coupling strength.
The degree  of nonlinearity is determined by $\sigma >0$. For $\gamma >0$ ($\gamma <0$) the nonlinearity is of focusing (defocusing) type and for $V_0>0$ ($V_0<0$) the delta potential is attractive (repulsive).
Without loss of generality we set $\kappa=1$.
In fact, by a scaling of time $\kappa t\mapsto \tilde{t}$ and the parameters
$\gamma/\kappa\mapsto \tilde{\gamma}$ and $V_0/\kappa\mapsto \tilde{V}_0$,  system becomes independent of $\kappa$. 
Note that by a
 transformation $\tilde{u}=\gamma^{1/(2\sigma)}u$  the nonlinearity parameter $\gamma$ could be absorbed into the amplitude. However, as in what follows we differentiate between focusing and defocusing nonlinearity,  we prefer to keep $\gamma$ in system (\ref{eq:system1}).

\begin{remark}
 The presence of the point defect breaks the (space) translational invariance. The system (\ref{eq:system1})  still exhibits  gauge invariance (i.e. multiplication by a complex phase),  and possesses the  time reversibility symmetry $t \leftrightarrow -t$, $u \leftrightarrow \overline{u}$ as well as time shift symmetry.
\end{remark}

$\overline{u}$ denotes the complex conjugate of $u$.

System (\ref{eq:system1}) has two conserved quantities, namely  the Hamiltonian (energy) 
\begin{equation}
 H(u)=\sum_{n\in {\mathbb{Z}^d}}\sum_{j}|u_{n+j}-u_{n}|^2-\frac{\gamma}{\sigma+1}\sum_{n\in {\mathbb{Z}^d}}|u_n|^{2(\sigma+1)}-V_0\sum_{n \in \mathbb{Z}^d} \delta_{n,0}|u_n|^{2}\label{eq:H}
\end{equation}
and the mass
\begin{equation}
 \nu=\sum_{n\in {\mathbb{Z}^d}}|u_n|^2\,.\nonumber
\end{equation}
The equations of motion (\ref{eq:system1}) are derived as 
\begin{equation}
 i\dot{u}_n=\frac{\partial H}{\partial \overline{u}_n}.\nonumber
\end{equation}

The DNLS with general power nonlinearity but without the delta potential, 
\begin{equation}
i\frac{d u_n}{dt}+(\Delta u)_n+|u_n|^{2\sigma}u_n=0,\,\,\,n \in \mathbb{Z}^d,\label{eq:DNLS0}
\end{equation}
has served as a model with numerous applications including optical lattices in Bose-Einstein condensates \cite{optics}, photonic crystals \cite{crystals}, arrays of optical fibers \cite{fibers}, plasmonic nanowires \cite{nanowires} and energy transport in biomolecules. For  reviews see \cite{Kevrekidis}-\cite{DNLS}. We emphasize that the DNLS (\ref{eq:DNLS0})
admits global solutions for all   $\sigma>0$ \cite{Pankov},\cite{Wang}.

We express (\ref{eq:system1}) in  operator form 
 \begin{equation}
  \frac{du}{dt}=i\left(\Delta_\delta u+F(u) \right),\label{eq:operator}
 \end{equation}
 where $\Delta_\delta=\Delta +V_\delta$ with $V_\delta:\,l^2(\mathbb{Z}^d) \mapsto l^2(\mathbb{Z}^d)$, $(V_\delta u)_n=V_0\delta_{n,0}u_n$ and $F(u)=\gamma |u|^{2\sigma}u$.
 The linear operator $\Delta_\delta$ is self-adjoint on $l^2(\mathbb{Z}^d)$ and $i\Delta_\delta \,:l^2(\mathbb{Z}^d) \mapsto l^2(\mathbb{Z}^d)$ is $\mathbb{C}-$linear and skew-adjoint and generates a group $(\mathcal{S}(t))_{t\in \mathbb{R}}$ of isometries on $l^2(\mathbb{Z}^d)$. ( Note that the delta potential acts as a (local) multiplication operator.) For fixed $T>0$ and initial data $u_0 \in l^2(\mathbb{Z}^d)$, a function $u \in C([0,T],l^2(\mathbb{Z}^d))$ is a solution of    (\ref{eq:operator}) if and only if
 \begin{equation}
  u(t)=\mathcal{S}(t)u_0+i\int_0^t \mathcal{S}(t-\tau)F(u(\tau))d\tau.\nonumber
 \end{equation}
In a similar vein as in  \cite{Pankov},\cite{Wang}, exploiting the conservation of energy and mass one proves, by the energy methods, that for any initial data $u_0\in l^2(\mathbb{Z}^d)$ there exists a unique global solution $u\in C^1(\mathbb{R},l^2(\mathbb{Z}^d))$ for any $\sigma >0$. In the forthcoming we write $l^2$ instead of $l^2(\mathbb{Z}^d)$.

It is illustrative to compare the global well-posedness of (\ref{eq:system1}) with data $u_0\in l^2$ for all powers of the nonlinear term $\sigma >0$
with that  of its continuous focusing counterpart
given by
\begin{equation}
 i\frac{\partial U}{\partial t}+\frac{\partial^2 U}{\partial x^2}+|U|^{2\sigma}U+V_0\delta U=0,\,\,\,U_0\in H^1(\mathbb{R}),\label{eq:NLSdelta}
\end{equation}
where $0\le \sigma < \infty$ and $\delta$ is the Dirac measure at the origin.
System (\ref{eq:NLSdelta}) has attracted considerable interest \cite{Caudrelier}-\cite{Ardila} from a physical as well as mathematical view point.  

Several studies have been conducted about the well-posedness of the
Cauchy problem, the existence of standing waves,  their orbital and strong stability, and their asymptotic behavior such as scattering, global existence and blow-up
(see for example \cite{LeCoz}-\cite{Banica}). In more detail, for $d=1$ global well-posedness of the Cauchy problem for (\ref{eq:NLSdelta}) with data in $H^1(\mathbb{R})$   was  proven for $0< \sigma <2$ in \cite{Fukuizumi1}.  
The ground state solution (minimizing  the associated action functional on $H^1(\mathbb{R})$) plays a crucial role as it provides a threshold for the scattering blow-up dichotomy result. For actions below and on that of the ground state solution the solutions are global and scatter in both time directions for negative virial functional.  The solutions blow up in both time directions for positive virial functional. 

With respect to nonlinear bound (solitary) states of (\ref{eq:NLSdelta}), $U(x,t)=\Phi(x)\exp(i \omega t)$,  from the comprehensive study in \cite{Fukuizumi1} we know that for $V_0>0$ the ground state is orbitally stable in $H^1(\mathbb{R})$ for any $\omega >V_0^2/4$ if $0<\sigma \le 2$. For $\sigma>2$  
there exists a critical frequency $\omega_c>V_0^2/4$ such that the ground state is stable  in $H^1(\mathbb{R})$ for any $V_0^2/4<\omega<\omega_c$ and unstable in $H^1(\mathbb{R})$ for any $\omega >\omega_c$. 

For $V_0<0$ and $\omega >V_0^2/4$ orbital stability in $H^1(\mathbb{R})$ of standing solitary waves $\Phi(x)\exp(i \omega t)$ was proven in \cite{Fukuizumi}, if $0<\sigma \le 1$
while their instability occurs
for any $\omega >V_0^2/4$ if  $\sigma \ge 2$.

Finally, let us comment
also on the behavior of the solutions to the NLS in absence of the delta potential, i.e
\begin{equation}
 i\frac{\partial U}{\partial t}+\frac{\partial^2 U}{\partial x^2}+|U|^{2\sigma}U=0,\label{eq:NLS0}
\end{equation}
for which the global existence of solutions for initial data in $H^1$ holds only for $\sigma <2$ while for $\sigma \ge 2$ blow up in finite time may occur (see e.g. in \cite{CazenaveLions}-\cite{Cazenavestab}).
  Regarding the orbital stability of standing solitary wave solutions $\Phi(x)\exp(i \omega t)$ of (\ref{eq:NLS0}), there are several studies \cite{CazenaveLions},\cite{Cazenave1},\cite{NLS01}-\cite{NLS03}. In \cite{CazenaveLions} it is shown that $\Phi(x)\exp(i \omega t)$ is stable in $H^1(\mathbb{R})$ for any $\omega>0$ if $1<\sigma <5$, and unstable in $H^1(\mathbb{R})$ for any $\omega>0$ if $\sigma \ge 5$.
In contrast to its continuum counterpart (\ref{eq:NLS0}), which is completely integrable for cubic nonlinearity, i.e. $\sigma=1$, \cite{Zakharov}-\cite{Newell}, the cubic DNLS is nonintegrable \cite{Ablowitz1},\cite{Levi}.

Similar to the dichotomy feature of the NLS with delta potential,  for the $L^2-$critical case $\sigma=2$, the solutions exist globally in $H^1(\mathbb{R})$ for subcritical mass $||U||_{L^2}<||\Phi||_{L^2}$ where $\Phi$ is the (radial) ground state and scatter to the linear solution $\exp(i \partial_x^2) U_0^{\pm}$ in both time directions \cite{Dodson}. For the critical mass $||U||_{L^2}=||\Phi||_{L^2}$, finite blow-up solutions occur \cite{Merle}. For super-critical mass $||U||_{L^2}>||\Phi||_{L^2}$,  we refer for blow-up results to \cite{Merle1}-\cite{Raphael}.
  
Concerning bound states of the lattice system (\ref{eq:system1}), we distinguish between linear and nonlinear bound ones, respectively. 
The former ones arise in the presence of a delta potential acting as an (external) impurity (defect) breaking the translational invariance of the lattice.  The origin of the latter is due to the focusing nonlinear term creating an intrinsic potential well on a part of the lattice in which excitation energy becomes trapped.
In this paper we study also the interplay between these two localization mechanisms. 

Our contributions in this paper are threefold.

First for focusing nonlinearity we study the existence of  nonlinear ground state  solutions corresponding to   stationary solutions $u_n(t)=x_n \exp(-i \omega t)$ with $\omega\in \mathbb{R}$.

Here we treat real-valued $x_n$. The stationary system
\begin{equation}
 \omega x_n+(\Delta x)_n+\gamma |x_n|^{2\sigma} x_n+V_0 \delta_{n,0} x_n=0,\nonumber
\end{equation}
possesses the conserved quantity $J=2Im(\overline{x}_nx_{n+1})$ \cite{PhysicsReport}. The decay property $x_n\rightarrow 0$ as $|n|\rightarrow \infty$ dictate $J=0$, from which in turn follows (excluding $x_n=0$)
\begin{equation}
 \frac{x_{n+1}}{x_n}=\frac{\overline{x}_{n+1}}{\overline{x}_n}.\nonumber
\end{equation}
With the presentation $x_n=r_n\exp(i\theta_n)$ one has
\begin{equation}
\exp(i\theta_{n+1})=\exp(i\theta_n).
\end{equation}
Hence, $x_{n+1}$ and $x_n$ have the same argument (modulo $\pi$) so that without loss of generality we may take $x_n$ real-valued for all $n\in \mathbb{Z}$.

These ground state solutions  are time-periodic and spatially exponentially localized (see below Proposition \ref{proposition:expdecay}); and are also referred to as {\it  breathers}.
The corresponding system of stationary equations reads
\begin{equation}
 \omega x_n+(\Delta x)_n+\gamma x_n^{2\sigma} x_n+V_0 \delta_{n,0} x_n=0,\,\,\,x_n\in \mathbb{R},\,\,\, n \in \mathbb{Z}^d,\label{eq:stat1}
\end{equation}
and can be derived from  the associated action functional $J\,:\,l^2(\mathbb{Z}^d)
 \mapsto \mathbb{R}$ given by
\begin{eqnarray}
 J(x)&=&
\frac{1}{2}\left\| \nabla x\right\|_{l^2}^2
-\frac{\omega}{2}\left\| x\right\|_{l^2}^2-\frac{\gamma}{2\sigma+2} \left\|x\right\|_{l^{2\sigma+2}}^{2\sigma+2}-\frac{V_0}{2}x_0^{2},\label{eq:J}
\end{eqnarray}
where $(\nabla x)_n=x_{n+j}-x_n$.

\begin{remark}
The frequency $\omega$ of a localized state $u_n(t)=x_n\exp(i\omega t)$
lies above respectively below the upper respectively lower edge of the absolutely continuous spectrum of $\Delta$. The latter is given by
$\sigma_{ac}=[0,4d]$. For $\omega<0$ the associated non-staggering localized state has lower value of action than its staggering counterpart with $\omega>4d$ (see e.g. \cite{PhysicsReport}).
\end{remark}

In this study we consider $\omega<0$ and $\omega>0$ will be treated elsewhere.

\begin{definition}\label{definition:M1}Minimization problem ({\bf{M1}}):
Let $\gamma>0$.
For given $\omega<0$ by an action ground state solution we understand $x\neq 0$ which minimizes the value of $J$ in the Nehari manifold   $\mathcal{N}=\{x\in l^2\setminus\{0\}\,:\,<J^\prime(x),x>=0 \}$.
\end{definition}

In particular, for $\gamma>0$ we are interested in the effect of a delta potential on the existence
of ground state solutions, both in the attractive and the repulsive case. To prove
the existence of such solutions, we use a special variational method, consisting in
minimizing the action functional under the Nehari manifold.  In the first step we prove in Proposition \ref{subsection:Nehari} that the Nehari manifold $\mathcal{N}$  is  nonempty.
The second step,  proves that a minimizer exists on $\mathcal{N}$.
Our main result is:
\begin{theorem}\label{theorem:solution}
Let $\gamma>0$.
 If $V_0<-\omega$ there exists a ground state with frequency $\omega$  in the sense of Definition \ref{definition:M1}.
\end{theorem}
In
the last step, it is verified that such solutions are critical points of the associated
action functional. 

\vspace*{0.5cm}

In \cite{thresh4} the existence of breathers of the DNLS with linear impurities was discussed.
The authors proved by variational methods the existence of ground states for {\it finite} lattices supplemented with Dirichlet boundary conditions. Due to the finite dimension of the associated functional space a straightforward application of the mountain pass theorem is possible.  In contrast, in the current paper the existence proof concerns {\it infinite} lattices for which, assuring that the  associated energy (action) functional exhibits the mountain pass geometry, would require further advanced techniques. Instead we use here a different variational approach, namely the Nehari manifold method. Nevertheless, our existence proofs are valid for finite lattices too.

\vspace*{0.5cm}

Our second result  concerns the existence of excitation thresholds for the creation of ground states with prescribed mass $\nu$. Such solutions are physically relevant because the mass is conserved throughout time.
Being interested in normalized solutions, that is solutions with  prescribed $l^2$ norm,  we consider a  solution $u_n(t)=x_n\exp(-i\omega t)$ to the stationary system (\ref{eq:stat1}) as  the pair $(x_\nu,\omega_\nu)\in l^2 \times \mathbb{R}$, with $\omega_\nu$ being the Lagrange multiplier to the critical point $x_\nu$ on $S_\nu$.
One looks for solutions with minimal energy constrained to the sphere $S_\nu$
\begin{equation}
 S_\nu=\left\{x\in l^2\,:\,||x||_{l^2}^2=\nu \right\}.
\end{equation}

\begin{definition}Minimization problem ({\bf{M2}}):
Let $\gamma>0$. A ground state with (prescribed) $\nu>0$ is defined as a minimizer of the energy functional
$E:\,l^2\mapsto \mathbb{R}$
\begin{equation}
E(x)=\frac{1}{2}\sum_{n \in \mathbb{Z}^d}\sum_j (x_{n+j}-x_{n})^2-\frac{\gamma}{2\sigma+2} \sum_{n \in \mathbb{Z}^d} x_n^{2\sigma+2}-\frac{1}{2}V_0\sum_{n \in \mathbb{Z}^d} \delta_{n,0}x_n^{2},\label{eq:energy}
 \end{equation}
constrained to $S_\nu$.
\end{definition}

The minimization problem {\bf M2} can be solved by an application of  concentration compactness methods \cite{Lions}, adapted to the discrete case \cite{Weinstein}. However, with Proposition \ref{proposition:equ} we establish equivalence between an action ({\bf M1}) respectively energy minimizing ({\bf M2}) ground state.
Hence, the proof of existence of ground state solutions associated
with minimization problem {\bf M1} proves at the same time the existence of ground state solutions associated with minimization problem {\bf M2}.

Crucially,  due to the $U(1)$ symmetry of the Hamiltonian $H(u)$ in (\ref{eq:H}), the total energy $H$ does not depend on the frequency $\omega$ of stationary solutions $u_n(t)=x_n\exp(-i\omega t)$ and $H(u)=2E(x)$, with $E$ given in (\ref{eq:energy}), and $||u||_{l^2}=||x||_{l^2}$. Therefore,
\begin{equation}
 \inf\left\{H(u)\,:\,||u||_{l^2}^2=\nu\right\}=\underline{H}_{\nu}=2\inf_{x\in S_\nu} E(x)=2\underline{E}_{\nu}.\label{eq:varprob}
 \end{equation}
Therefore, we may use  Weinstein's criterion
([\cite{Weinstein}, definition pg. 678]) stating  that if a ground state (as a minimizer of the constrained variational problem (\ref{eq:varprob}) exists for any value of $\nu>0$, there is no excitation threshold. Conversely, if there exists a positive $\nu_{exc}$ such that $\underline{H}_\nu<0$ if and only if $\nu>\nu_{exc}$, then $\nu_{exc}$ constitutes an excitation threshold.
In this context, it was proved in \cite{Weinstein} that,
if $\sigma\ge 2/d$ the DNLS with general power nonlinearity
\begin{equation}
i\frac{d u_n}{dt}+(\Delta u)_n+\gamma|u_n|^{2\sigma}u_n=0,\,\,\,n \in \mathbb{Z}^d,\,\,\,\gamma>0,\label{eq:system2}
\end{equation}
exhibits ground state  solutions if and  only if the total mass $\nu$ is larger than some strictly positive threshold value $\nu_{exc}$.

\begin{theorem}\label{theorem:threshold}
Consider equation (\ref{eq:system1}) with  general power nonlinearity $\sigma>0$ and $\gamma>0$.
 
 \noindent 1) For  attractive delta potential $V_0>0$ and all $\sigma>0$  there is no excitation threshold, i.e. $\underline{E}_\nu<0$ for all $\nu>0$.

 \noindent 2) For repulsive delta potential $V_0<0$ and $\sigma<\min\{1,2/d\}$ there is no excitation threshold, i.e. $\underline{E}_\nu<0$ for all $\nu>0$.
For $\sigma\ge \max\{1,2/d\}$ a ground state solution exists if and only if
 \begin{equation}
 \nu>\nu_{exc}=\left(\frac{(\sigma+1)(2-V_0)}{\gamma}\right)^{1/\sigma}>0.
\end{equation}
 \end{theorem}

 The additional contribution of an attractive delta potential in the focusing DNLS prevents the existence of an excitation threshold for ground states. In other words, for all $\nu>0$ and all $\sigma >0$ there exist ground state solutions. Indeed, in this case the linear and nonlinear localization mechanisms conspire lowering the total energy.

For focusing nonlinearity $\gamma>0$, repulsive delta potential $V_0<0$, and supercritical value of the degree of  nonlinearity, i.e. $\sigma>\min\{1,2/d\}$, the condition $\nu>\nu_{exc}>0$ is necessary and sufficient for the existence of ground states, i.e. $\underline{H}_\nu=2\underline{E}_\nu<0$. When the degree of the nonlinearity is subcritical, i.e. $\sigma<\min\{1,2/d\}$,  ground state solutions are supported for any mass. Consequently,  ground states of arbitrarily small $l^2$ norm exist.
Note that a repulsive delta potential counteracts the nonlinear localization mechanism. Remarkably, there exist  localized states with
frequencies above respectively below the upper respectively lower edge of absolutely continuous spectrum of $\Delta$. Interestingly, as will be further explained at the end of  Section \ref{section:threshold} in  Remark \ref{remark:simultaneous}
these localized states have the form of a staggering defect mode respectively breather simultaneously in the same system.
This is a repercussion of the spatial discreteness of the system, and therefore not present in the continuum system. 

Conversely,  defocusing nonlinearity weakens the localization capability of an attractive delta potential. That is,  there is an {\it upper} threshold, confining the $l^2$-size of the nonlinear term, up to which   the linear bound state persists.

\vspace*{0.5cm}

We emphasize that our excitation thresholds  $\nu_{exc}$  for the formation of breathers differ from those, $\nu_{crit}$,  considered in \cite{thresh4},\cite{thresh2},\cite{thresh3}. To be precise, explicit lower bounds for the energy (interpreted as the mass, i.e. the $l^2$ norm) of breathers are presented. These bounds depend on the dimension, the parameters of the DNLS and the frequency of the breathers. Their relevance is that
for a set of given parameters of the DNLS and the breather frequency they determine lower bounds on the power below which nontrivial breathers cannot exist.

Moreover, $\nu_{crit}$ accounts for {\it any}
 solution $u_{1\le n\le N}(t)=x_{1\le n\le N}\exp(i\omega t)\in l^2(\mathbb{Z}_N)$ of the associated stationary system (e.g. multi-hump states and not necessarily the ground state(s)). In comparison,  our excitation thresholds  $\nu_{exc}$ provide  stronger results as they give  thresholds for the existence of {\it ground states}. In fact, the thresholds
$\nu_{crit}$ \cite{thresh4},\cite{thresh2},\cite{thresh3}
are  not sharp in the sense of a  threshold value for the existence--nonexistence of breathers as $\nu_{exc}>\nu_{crit}$ \cite{thresh4},\cite{thresh2}.  Notice that,
while  the excitation threshold $\nu_{crit}$ in \cite{thresh4} determines for {\it any} dimension the smallest value of the power   breathers can have,  the excitation thresholds dealt with in  \cite{Weinstein} as well as in the current paper  occur only for overcritical dimensions $d\ge d_{crit}$ as stated in Theorem \ref{theorem:threshold}.

Our  third main result contained in Section \ref{section:scattering}
concerns the asymptotic properties of the solutions to system (\ref{eq:system1}).

 \begin{theorem}\label{theorem:scatterin1}
Consider system (\ref{eq:system1}) with focusing nonlinearity $\gamma>0$ and $V_0=0$.
 For  $\sigma \ge \max\{1,2/d\}$  let $2(\sigma+1)\ge p\ge 4d\sigma/(2d\sigma-3)$ and
 $u\in C([0,\infty),l^2)$ be a solution to (\ref{eq:system1}) for initial data
 $u_0\in l^2$ such that $||u_0||_{l^2}^2\le \nu_{exc}$.
 Then there exists  $v^{\pm} \in l^{p}$, $p>2$, such that
 \begin{equation}
  \left\|\left(\exp(-i\Delta t)\right)u(t)-v^{\pm} \right\|_{l^p}\le C \left\| u_0\right\|_{l^{p^\prime}}^{2\sigma+1}\,
  |t|^{-\frac{d(p-2)}{3p}},\,\,\,\,\forall |t| > 0.\nonumber
 \end{equation}
 The solution $u(t)$ scatters forward and backward in time in the sense that there exists $v^{\pm}\in l^2$ such that
 \begin{equation}
 \left\|u(t)-\left(\exp(i\Delta t)\right)v^{\pm}  \right\|_{l^p}\longrightarrow 0\,\,\,{\rm as}\,\,\,
 t\rightarrow  \pm \infty.\nonumber
\end{equation}
 \end{theorem}
The proof of the assertions in Theorem \ref{theorem:scatterin1}  establishes that
the solutions of the nonlinear problem  exhibit asymptotically free behavior. Notably, the asymptotic features of the solutions  prove also Weinstein's conjecture in \cite{Weinstein}
for $\sigma \ge 2$: If $\nu(u_0)=||u_0||_{l^2}^2<\nu_{exc}$, then for any $p\in (2,\infty]$ the solutions decay, that is, $ \left\| u(t)\right\|_{l^p}\rightarrow 0,\,\,\,{\rm as}\,\,\,|t|\rightarrow \infty$. More details about the proof of the conjecture are given at the end of Section 5.

For $V_0 <0$ we have the following statement:
\begin{theorem}\label{theorem:scatterin2}
For $d=1$ consider system (\ref{eq:system1}) with  focusing nonlinearity $\gamma>0$ and $V_0<0$. For $\sigma >9/4$  let $2(\sigma+1)\ge p\ge 4\sigma/(2\sigma-3)$ and
 $u\in C([0,\infty),l^2)$ be a solution to (\ref{eq:system1}) for initial data
 $u_0\in l^2$ such that $||u_0||_{l^2}^2\le \nu_{exc}$.
 Then there exists  $v^{\pm} \in l^{p}$, $p>2$, such that
 \begin{equation}
  \left\|\left(\exp(-i\Delta_\delta t)\right)P_{ac}u(t)-v^{\pm} \right\|_{l^p}\le C \left\| u_0\right\|_{l^{p^\prime}}^{2\sigma+1}\,
  |t|^{-\frac{p-2}{3p}},\,\,\,\,\forall |t| > 0,\nonumber
 \end{equation}
 where $P_{ac}$ is the projection onto the continuous spectral subspace of $\Delta_\delta$.
 The solution $u(t)$ scatters forward and backward in time in the sense that there exists $v^{\pm}\in l^2$ such that
 \begin{equation}
 \left\|u(t)-\left(\exp(i\Delta_\delta t)\right)P_{ac}v^{\pm}  \right\|_{l^p}\longrightarrow 0\,\,\,{\rm as}\,\,\,
 t\rightarrow  \pm \infty.\nonumber
\end{equation}
 \end{theorem}

 \vspace*{0.5cm}

For a summary of the  distinction between the study of the DNLS with linear impurities \cite{thresh4} and  our present one of the $\delta$DNLS (\ref{eq:system1}), we note that  our investigations go beyond those in \cite{thresh4}
as (i) our analysis is not restricted to finite lattices but rather deals with infinite lattices, (ii) our
excitation thresholds concern ground state solutions as opposed to any
solution in $l^2$ in \cite{thresh4} which are not necessarily ground states, (iii) we present here a detailed analysis of the interplay between the two localization mechanisms, being of linear respectively nonlinear nature, inherent in the $\delta$DNLS (\ref{eq:system1}), and (iv), asymptotic behavior was not considered  at all in \cite{thresh4}.  Most of our analysis concerns focusing nonlinearity $\gamma>0$. The impact of nonfocusing nonlinearity $\gamma<0$ on the persistence of linear impurity modes is then briefly discussed at the end of Section 4.

The paper is organized as follows: 
We begin with a brief summary of linear defect modes and the impact of weak nonlinearity on their persistence. In section \ref{section:breathers}, using variational methods,  we prove the existence of  standing ground state solutions for focusing nonlinearity under the influence of an attractive respectively repulsive delta potential.
Section \ref{section:threshold} is concerned with the existence of an excitation threshold for the presence of ground states respectively   perseverance of defect modes. Subsequently, in section \ref{section:scattering} we investigate the asymptotic behavior of the solutions and  prove that they scatter for supercritical degree of nonlinearity and subcritical mass.

\section{Linear defect modes}\label{section:linear}
\setcounter{equation}{0}
In the linear limiting case, arising for $\gamma=0$ in (\ref{eq:system1}), the system 
\begin{equation}
i\frac{d v_n}{dt}+(\Delta v)_n+V_0 \delta_{n,0}v_n=0,\,\,\,n \in \mathbb{Z}^d.\label{eq:systemlin}
\end{equation}
supports localized solutions (also called defect modes) generated by the presence of the point defect destroying the  spatial translational invariance. 
These stationary solutions are of the form $v_n(t)=x_n \exp(-i \omega t)$, $x_n\in \mathbb{R}$ with a spectral parameter (eigenenergy) $\omega$ and fulfil the stationary system
\begin{equation}
 \omega x_n+(\Delta x)_n+V_0 \delta_{n,0} x_n=0,\,\,\,n \in \mathbb{Z}^d.\label{eq:statlin}
\end{equation}
For $V_0>0$  there exist a  non-staggering  bound state
\begin{equation}
 x_n=A\eta_+^{|n|},\,\,\,0<\eta_+ <1,\nonumber
\end{equation}
\begin{equation}
 \eta_+=\frac{{1}}{2d}\left(\sqrt{{V_0}^2+4d^2}-{V_0}\right),\nonumber
\end{equation}
with $|n|=|n_1|+...+|n_d|$ and amplitude $A>0$.

For $V_0<0$ there exist a staggering bound state
\begin{equation}
 x_n=A (-1)^n \eta_-^{|n|},\,\,\,-1<\eta_-<0,\nonumber
\end{equation}
\begin{equation}
\eta_-=-\frac{{1}}{2d}\left(\sqrt{{V_0}^2+4d^2}+{V_0}\right),\nonumber
\end{equation}
with $|n|=|n_1|+...+|n_d|$ and amplitude $A>0$.

The amplitude $A$ is related to the conserved $l^2$ norm
(mass)
\begin{equation}
 \nu=\sum_{n \in \mathbb{Z}^d}|x_n|^2=\left(A^2\frac{1+\eta_{\pm}^2}{1-\eta_{\pm}^2}\right)^d.\nonumber
\end{equation}

The associated eigenvalues are determined by
\begin{equation}
 {\omega}_+=2d- \sqrt{{V_0}^2+4d^2},\,\,\,\left( {\omega}_-=2d+ \sqrt{{V_0}^2+4d^2}\right).\label{eq:eigenvalue0}
\end{equation}
The eigenvalues  $\omega_+<0$ ($\omega_->4d$) lie below (above) the lower edge of absolutely continuous spectrum of $\Delta$.  
Note that the defect modes are exponentially localized at the origin of the lattice (single-hump).

Next, we briefly discuss the impact of a (small) nonlinear term on the impurity modes. To be precise, we consider the system
\begin{equation}
i\frac{dw_n }{dt}+(\Delta w)_n+ \gamma|w|^{2\sigma}w_n+V_0 \delta_{n,0}w_n=0,\,\,\,n \in \mathbb{Z}^d.\label{eq:systemnonlin}
\end{equation}
Obviously, if $\gamma>0$ and $V_0>0$ respectively $\gamma<0$ and $V_0<0$, the linear and nonlinear localization mechanisms enhance each other opposed to the case if  $\gamma$ and $V_0$ are of opposite sign.
In the latter case, if $V_0>0$ ($V_0<0$) an additional  negative (positive) nonlinear term, i.e. $\gamma<0$ ($\gamma>0$), in the equation of motion leads to an increase (decrease) of the negative (positive) total energy (\ref{eq:H}) of the impurity mode leading to  a reduction of its  degree of localization.
Nevertheless, for sufficiently low amplitude ($l^2$ norm), going along with weak nonlinearity, we expect sustained  impurity modes even when $V_0$ and $\gamma$ have different sign.

More specifically, for an attractive delta potential, when the non-staggering defect mode has 
negative energy, an additional  positive nonlinear term $\gamma|u|^{2\sigma+1}>0$ in the equations of motion serves constructively 
to even lower the total energy whereas negative nonlinearity for $\gamma<0$  leads to an increase of 
the total energy. The energy remains negative though if the nonlinear term is sufficiently small. 
Similar arguments hold for the staggering defect mode.
Therefore, we expect that there are critical values of the mass that the defect modes can sustain. 
In fact, in section \ref{section:threshold} we further analyze the impact of nonlinearity  on the
existence of defect mode(s) where we also derive explicit formulaes for upper and lower
thresholds of the mass for which the defect modes survive in the presence of nonlinearity.

\section{Ground state solutions}\label{section:breathers}
In this section we consider ground states of system
(\ref{eq:system1}) with focusing nonlinearity as the solutions of minimization problem {\bf M1}.

\vspace*{0.5cm}

Concerning  the action functional (\ref{eq:J}), clearly $J(x)\in C^1(l^2,\mathbb{R})$ and
\begin{equation}
 <J^\prime(x),z)>=
 \sum_{n \in \mathbb{Z}^d} (\nabla x)_n (\nabla z)_n
 -\omega \sum_{n \in \mathbb{Z}^d}  x_n z_n-\gamma (x^{2\sigma+1},z)_{l^2}-V_0\sum_{n \in \mathbb{Z}^d} \delta_{n,0} x_n z_n,\nonumber
\end{equation}
where  $z\in l^2$ and $(\,,\,)_{l^2}$  denotes the  scalar product $(x^{(1)},x^{(2)})_{l^2}=\sum_n  x^{(1)}_n {x}_n^{(2)}$ for all $x^{(1)},x^{(2)} \in l^2$.
Note that $J$ is not necessarily bounded from below on $l^2$. However, $J$ can be rendered  bounded from below on a suitable subset of $l^2$, namely the so-called Nehari manifold.

In detail, we introduce the Nehari functional
\begin{equation}
 I(x)=<J^\prime(x),x>=
 \sum_{n \in \mathbb{Z}^d} (\nabla x)_n^2
 -\omega \sum_{n \in \mathbb{Z}^d}   x_n^2-\gamma \sum_{n \in \mathbb{Z}^d}x_n^{2(\sigma+1)}  -V_0\sum_{n \in \mathbb{Z}^d} \delta_{n,0}  x_n^{2},\nonumber
\end{equation}
and the Nehari manifold
\begin{equation}
 \mathcal{N}=\left\{x\in l^2:\,I(x)=0, x\neq 0 \right\}. \label{eq:Nehari}
\end{equation}
$I(x) \in C^1(l^2,\mathbb{R})$ and
\begin{eqnarray}
 <I^\prime(x),z>&=&
 2\sum_{n \in \mathbb{Z}^d} (\nabla x)_n (\nabla z_n)
 -2\omega\sum_{n \in \mathbb{Z}^d}  x_n z_n-2(\sigma+1)\gamma \sum_{n \in \mathbb{Z}^d}  x_n^{2\sigma+1} z_n\nonumber\\
 &-&2V_0\sum_{n \in \mathbb{Z}^d} \delta_{n,0}  x_n z_n.\nonumber
\end{eqnarray}

\vspace*{0.5cm}

We define the fibering map $\Gamma_x(s)=J(sx)$ for $s>0$. Obviously, if $x$ is a local minimizer of $J$ on $\mathcal{N}$, then $\Gamma_x$ possesses a local minimum at $s=1$. Furthermore, for $x \in l^2\setminus \{0\}$, $sx \in \mathcal{N}$ if and only if $\Gamma_x^\prime(s)=0$, a result that follows readily from the fact that $\Gamma_x^\prime(s)=<J^\prime(sx),x>=(1/s)<J^\prime (sx),sx>$. Conclusively, stationary points of the map $\Gamma_x$ correspond to points in $\mathcal{N}$.

\noindent
In the forthcoming we make use of the following theorem relating minimizers on $\mathcal{N}$ and critical points for $J$:

\begin{theorem}\label{theorem:minimizer}
 If  $x_0$ is a local minimizer for $J$
 on $\mathcal{N}$
 and  $<I^\prime(x_0),x_0>\neq 0$, then $J^\prime(x_0)=0$.
\end{theorem}

\begin{proof} $x_0$ being a minimizer of $J$ on $\mathcal{N}$
means that by the method of Lagrange multipliers there exists a $\lambda \in \mathbb{R}$ such that $J^\prime(x_0)=\lambda I^\prime(x_0)$. Hence,
\begin{equation}
<J^\prime(x_0),z)>+\lambda <I^\prime(x_0),z>=0,\label{eq:infer}
\end{equation}
for any $z \in l^2$, yielding for $z=x_0$
\begin{equation}
 I(x_0)+\lambda <I^\prime(x_0),x_0>=0.\nonumber
\end{equation}
As $x_0 \in \mathcal{N}$, $I(x_0)=0$.
Furthermore, since $<I^\prime(x_0),x_0>\neq 0$ it follows that $\lambda=0$ and the proof is finished.
\end{proof}

\begin{remark}
 $x_0$ is a weak solution of (\ref{eq:stat1}) because from (\ref{eq:infer}) we infer that $<J^\prime(x_0),z>=0$ for any $z \in l^2$.
\end{remark}

\vspace*{0.5cm}

\subsection{Nonempty Nehari manifold}\label{subsection:Nehari}

\begin{proposition}\label{proposition:Nehari}
Assume $\gamma > 0$. Let $0>\omega$ and
\begin{equation}
V_0<-\omega.
\label{eq:ass1}
\end{equation}
Then the Nehari manifold $\mathcal{N}\in C^1(l^2,\mathbb{R})$ is  nonempty.
\end{proposition}

\begin{proof}
We have
\begin{eqnarray}
I(sx)&=&
s^2\left(\sum_{n \in \mathbb{Z}^d} (\nabla x)_n^2
-\omega \sum_{n \in \mathbb{Z}^d}  x_n^2-V_0 x_0^{2}\right)-s^{2(\sigma+1)} \gamma \sum_{n \in \mathbb{Z}^d}x_n^{2(\sigma+1)}\nonumber\\
&\equiv&s^2A_x-s^{2(\sigma+1)}B_x,
\end{eqnarray}
where $B_x>0$ and due to the assumptions $\omega<0$ and $V_0<-\omega$ it holds  $A_x>0$. Hence, there must exist an $\tilde{s}>0$ such that $I(\tilde{s}x)=0$. Therefore $\tilde{s}x \in \mathcal{N}$. Thus, the Nehari manifold is nonempty.
\end{proof}

\vspace*{0.5cm}

\subsection{Existence of  ground state solutions}\label{subsection:existence}

\begin{proposition}\label{proposition:gtzero}
 If $V_0<-\omega$, then for every $x\in \mathcal{N}$, one has
\begin{eqnarray}
 ||x||_{l^2}\ge \alpha>0,\nonumber\\
 J(x)\ge \beta >0.\nonumber
\end{eqnarray}
 \end{proposition}

\begin{proof}
 $I(x)=0$ implies $-\omega ||x||_{l^2}^2=
 -||\nabla x||_{l^2} +\gamma ||x||_{l^{2(\sigma+1)}}^{2(\sigma+1)}+V_0x_0^2$.
Whence,
\begin{eqnarray}
-\omega ||x||_{l^2}^2&\le&\gamma ||x||_{l^{2(\sigma+1)}}^{2(\sigma+1)}+V_0x_0^2\nonumber\\
&\le& \gamma ||x||_{l^2}^{2(\sigma+1)}+V_0||x||_{l^2}^2,\nonumber
\end{eqnarray}
from which follows
\begin{equation}
 ||x||_{l^2}\ge \left(\frac{-(\omega+V_0)}{\gamma}\right)^{1/(2\sigma)}:=\alpha>0.
\end{equation}
Furthermore,
\begin{eqnarray}
 J(x)&=& \frac{1}{2}\left(1-\frac{1}{1+\sigma} \right)\left[
 ||\nabla x||_{l^2}^2
 - \omega ||x||_{l^2}^2-V_0x_0^2\right]\nonumber\\
 &\ge& \frac{1}{2}\left(1-\frac{1}{1+\sigma} \right)\left[-\omega ||x||_{l^2}^2-V_0x_0^2\right]\nonumber\\
 &\ge& \frac{1}{2}\left(1-\frac{1}{1+\sigma} \right)(-(\omega+V_0) ||x||_{l^2}^2\nonumber\\
 &\ge& \frac{1}{2}\left(1-\frac{1}{1+\sigma} \right)\left(-(\omega+V_0)\right)^{1+1/(2\sigma)} \gamma^{-1/(2\sigma)}:=\beta,
 \end{eqnarray}
 and the proof is finished.
\end{proof}

\begin{remark}\label{remark:bounded}
 Let $\{x_k\}\in \mathcal{N}$ be a minimizing sequence.
 Lemma \ref{proposition:gtzero} bounds the $J(x_k)$ uniformly, i.e. $J(x_k)\ge \beta>0$ for all $k\in \mathbb{N}$,  implying
 \begin{equation}
  \lim_{k\rightarrow \infty}J(x_k)=\inf_{x\in \mathcal{N}}J(x)>0.\label{eq:Jgtzero}
 \end{equation}
 Furthermore,   from
 \begin{eqnarray}
  J(x_k)\ge \frac{1}{2}\left(1-\frac{1}{1+\sigma} \right)(-(\omega+V_0) ||x_k||_{l^2}^2\nonumber
 \end{eqnarray}
 follows
 \begin{equation}
 ||x_k||_{l^2}\le \left(  \left(\frac{1}{2}-\frac{1}{2\sigma+2}\right)(-(\omega+V_0)\right)^{-1}S,
\end{equation}
where $S=\sup_{x\in \mathcal{N}}J(x)$. As $J(x)\le (4d-\omega-V_0)||x||_{l^2}^2=(4d-\omega-V_0) \nu$, $\omega<0$, and $V_0<-\omega$, we infer $S<\infty$ for $x\in \mathcal{N}$.
Thus, $\{x_k\}$ is bounded in $l^2$.
\end{remark}

\vspace*{0.5cm}

We are ready to prove Theorem \ref{theorem:solution}.

\begin{proof}
In order to prove the assertions of Theorem \ref{theorem:solution} we establish that there exists a minimizer for $J$ on $\mathcal{N}$ which is a critical point of $J(x)$ and thus a nontrivial solution of (\ref{eq:stat1}).
Note that for all $x \in l^2$, there exists an $a>0$ such that $-(\Delta x,x)_{l^2}-\omega \sum_{n \in \mathbb{Z}^d}  x_n^2\ge a ||x||_{l^2}^2$.

\vspace*{0.5cm}

\noindent 1)
Attractive delta potential $V_0>0$.

Let $\{x_k\}\in \mathcal{N}$ be a minimizing sequence. That is,
$\lim_{k\rightarrow \infty}J(x_k)=\inf_{x\in \mathcal{N}}J(x)$.
As noted in Remark \ref{remark:bounded},  $\{x_k\}$ is bounded in $l^2$.
  Hence, there exists a subsequence (not relabeled) that converges weakly in $l^2$, i.e. $x_k \rightharpoonup \tilde{x}$.
 In order to show  strong convergence, $x_k \rightarrow \tilde{x}$ in $l^2$, we assume  for a  contradiction
that
\begin{equation}
||\tilde{x}||_{l^2}<\underline{\lim}_{k\rightarrow \infty}||x_k||_{l^2}.\label{eq:contra}
\end{equation}
Then
\begin{eqnarray}
&&||\nabla \tilde{x}||_{l^2}^2-\omega ||\tilde{x}||_{l^2}^2-\gamma ||\tilde{x}||_{l^{2(\sigma+1)}}^{2(\sigma+1)}-V_0\tilde{x}_0^2\nonumber\\
 &<&\lim_{k \rightarrow \infty}||\nabla x_k||_{l^2}^2-\omega ||x_k||_{l^2}^2-\gamma ||x_k||_{l^{2(\sigma+1)}}^{2(\sigma+1)}-V_0x_{0,k}^2
=0.\nonumber
\end{eqnarray}
Facilitating the map $\Gamma_{{x}}$, satisfying $\Gamma_{{x}}(0)=0$, and
the proof of Proposition \ref{proposition:Nehari}, we have  $\Gamma_{{x}}(s)>0$
for sufficiently small $s$ and $\Gamma_{{x}}(s) \rightarrow -\infty$ for $s \rightarrow \infty$, and $\Gamma_{{x}}(s)$ possesses a unique maximum at $s(x)$ and $s(x)x \in \mathcal{N}$.
One gets
\begin{eqnarray}
 \Gamma_{\tilde{x}}^\prime(1)&=&||\nabla \tilde{x}||_{l^2}^2-\omega ||\tilde{x}||_{l^2}^2-\gamma ||\tilde{x}||_{l^{2(\sigma+1)}}^{2(\sigma+1)}-V_0\tilde{x}_0^2
 <0.\nonumber
\end{eqnarray}
Hence, there is an $0<s_0<1$ such that $\Gamma_{\tilde{x}}^\prime(s_0)=0$,  that is
$s_0 \tilde{x} \in \mathcal{N}$. Furthermore, $s_0 x_k \rightharpoonup  s_0 \tilde{x}$ in $l^2$,
and as $x_k \in \mathcal{N}$, the map $\Gamma_{x_k}$ attains its maximum at $s=s_0$.
Therefore,
\begin{equation}
 \Gamma_{\tilde{x}}(s_0)=J(s_0 \tilde{x})<\underline{\lim}_{k\rightarrow \infty}J(s_0x_k)\le \lim_{k\rightarrow \infty}J(x_k)=\inf_{x\in \mathcal{N}}J(x),\nonumber
\end{equation}
 contradicting our assumption (\ref{eq:contra}) and we conclude  $x_k\rightarrow \tilde{x}$ in $l^2$.
It then follows $
||\nabla \tilde{x}||_{l^2}^2
-\omega ||\tilde{x}||_{l^2}^2-\gamma ||\tilde{x}||_{l^{2(\sigma+1)}}^{2(\sigma+1)}-V_0\tilde{x}_0^{2}=0$ implying $\tilde{x} \in \mathcal{N}$.
Furthermore,
\begin{eqnarray}
0\stackrel{(\ref{eq:Jgtzero})}{<}\inf_{x\in \mathcal{N}}J(x)=\lim_{k \rightarrow \infty} J(x_k)&=&\left( \frac{1}{2}-\frac{1}{2\sigma+2}\right)\gamma\lim_{k \rightarrow \infty}
 ||x_k||_{l^{2(\sigma+1)}}^{2(\sigma+1)}\nonumber\\
&=&\left( \frac{1}{2}-\frac{1}{2\sigma+2}\right)\gamma ||\tilde{x}||_{l^{2(\sigma+1)}}^{2(\sigma+1)}= J(\tilde{x}),\nonumber
\end{eqnarray}
so that $\tilde{x}$ is a minimizer on  $\mathcal{N}$.
Finally, in accordance with Theorem \ref{theorem:minimizer}, we rule out that the map $\Gamma_x$ possesses inflection points requiring $\Gamma^{\prime\prime}(1)=0$.
We obtain
\begin{eqnarray}
 \Gamma_x^\prime(s)&=&<J^\prime(sx),x>\nonumber\\
 &=&
- \sum_{n \in \mathbb{Z}^d} (s(\nabla x)_n)(\nabla x)_n
- \omega \sum_{n \in \mathbb{Z}^d} ({sx}_n) x_n-\gamma \sum_{n \in \mathbb{Z}^d}(sx_n)^{2\sigma+1}x_n-V_0(sx_0) x_0,\nonumber
\end{eqnarray}
and
\begin{eqnarray}
 \Gamma_x^{\prime \prime}(s)&=&<J^{\prime \prime}(sx),x>\nonumber\\
 &=& \gamma (1-(2\sigma+1) s^{2\sigma})\sum_{n \in \mathbb{Z}^d}(sx_n)^{2\sigma+1}x_n
 +V_0(1-s)x_0^{2},\nonumber
\end{eqnarray}
yielding  $\Gamma_{\tilde{x}}^{\prime \prime}(1)=-2\sigma\gamma \sum_{n \in \mathbb{Z}^d}x_n^{2\sigma+2}<0$. That is, due to Theorem \ref{theorem:minimizer}, $\tilde{x}$ is a critical point of $J$ and the proof for $V_0>0$ and $\gamma>0$ is complete.

\vspace*{0.5cm}

\noindent 2) Repulsive delta potential $V_0<0$.

Take  a minimizing sequence, $\{x_k\}\in \mathcal{N}$, i.e.
$\lim_{k\rightarrow \infty}J(x_k)=\inf_{x\in \mathcal{N}}J(x)$.
By Remark \ref{remark:bounded}, $\{x_k\}$ is bounded in in $l^2$, so that there exists a subsequence (not relabeled) that converges weakly in $l^2$, i.e. $x_k \rightharpoonup \tilde{x}$. The proof that $x_k$ converges strongly to $\tilde{x}$ in $l^2$ proceeds analogously to the one in case 1) $V_0>0$ above.
From
$-(\Delta \tilde{x},\tilde{x})_{l^2}-\omega \sum_{n \in \mathbb{Z}^d}    \tilde{x}_{n \in \mathbb{Z}^d}^2-\gamma (f(\tilde{x}),\tilde{x})_{X}+V_0\tilde{x}_0^{2}=0$ follows $\tilde{x} \in \mathcal{N}$.
In addition, $J(\tilde{x})=\lim_{k\rightarrow \infty} J(x_k)=\inf_{x\in \mathcal{N}}J(x)$. Hence, $\tilde{x}$ is a minimizer on  $\mathcal{N}$.

Finally, we show that $\Gamma^{\prime\prime}(1)=0$ establishing that the map $\Gamma_x$ does not have inflection points.
We have
\begin{eqnarray}
 \Gamma_x^\prime(s)&=&<J^\prime(sx),x>\nonumber\\
 &=&
 \sum_{n \in \mathbb{Z}^d} (s(\nabla x)_n)(\nabla x)_n
 - \omega \sum_{n \in \mathbb{Z}^d}  ({sx}_n) x_n-\gamma (f({sx}),{x})_{X}-V_0(sx_0) x_0,\nonumber
\end{eqnarray}
and
\begin{eqnarray}
 \Gamma_x^{\prime \prime}(s)&=&<J^{\prime \prime}(sx),x>\nonumber\\
 &=& \gamma (1-(2\sigma+1) s^{2\sigma})||x||_{l^{2(\sigma+1)}}^{2(\sigma+1)}+V_0(1-s)x_0^{2},
\end{eqnarray}
yielding  $\Gamma_{\tilde{x}}^{\prime \prime}(1)=-2\sigma\gamma ||\tilde{x}||_{l^{2(\sigma+1)}}^{2(\sigma+1)}=-2\sigma\gamma \sum_{n \in \mathbb{Z}^d} \tilde{x}_n^{2\sigma+2}<0$. That is, by  Theorem \ref{theorem:minimizer}, $\tilde{x}$ is a critical point of $J$ and the proof for $V_0<0$ and $\gamma>0$ is finished.
\end{proof}

\section{Excitation thresholds}\label{section:threshold}
In this section we investigate the impact of  nonlinearity on the durability of the (linear) defect modes and, reciprocally, how the delta potential affects the existence of (nonlinear) ground state solutions.

In the previous section we established  the existence of   stable ground state solutions  as the critical point in $l^2$ of the action  functional
\begin{equation}
 J(x)=-\frac{1}{2}(\Delta x,x)_{l^2}
 -\frac{\omega}{2}\sum_{n \in \mathbb{Z}^d} x_n^2-\frac{\gamma}{2\sigma+2} \sum_{n \in \mathbb{Z}^d} x_n^{2\sigma+2}-\frac{V_0}{2}\sum_{n \in \mathbb{Z}^d} \delta_{n,0}x_n^{2},\label{eq:action}
\end{equation}
where $\omega$ is treated as a parameter.
However, nothing can be stated about the value of the $l^2$ norm of the solutions.

The influence of the $l^2$ norm on the existence of  ground state solutions was studied in \cite{Weinstein}. The latter are determined by   the critical points (minimizers) of the energy functional $H(u)$ given in (\ref{eq:H}) with a prescribed $l^2$ norm (mass), i.e. solutions
 constrained to $l^2-$spheres $S=\left\{u\in C^1([0,\infty);l^2)\,:\,||u||_{l^2}^2=\nu \right\}$.

Before we proceed with our study of excitation thresholds,
  we prove the  equivalence between an action ({\bf M1}) respectively energy minimizing ({\bf M2}) ground state following the arguments in \cite{equ1},\cite{equ2}:
\begin{proposition}\label{proposition:equ}
Let $\omega(y)$ be the Lagrange multiplier associated with an arbitrary minimizer
$y\in S_\mu$ of (M2). Then,
for given
$\omega \in \{\omega(y)\,:\,y\in S_\nu\,\,{\text is\,\,a\,\,minimizer\,\,of\,\,(M1)}\}$,
any  ground state solution $x\in l^2$
of (M1) is  a  minimizer of (M2), that is,
\begin{equation}
 ||x||_{l^2}^2=\nu\,\,\,{\text and}\,\,\,E(x)=\underline{E}_\nu.
\end{equation}
\end{proposition}

\begin{proof}
The proof comprises three steps:

\underline{Claim 1:} For every $x\in l^2\setminus\{0\}$ there exists a unique $s(x)>0$ such that  $s(x)x\in \mathcal{N}$ and
\begin{equation}
 \max_{s>0}J(sx)=J(s(x)x).\label{eq:claim1}
\end{equation}

\vspace*{0.3cm}

Indeed,
for fixed $x\in l^2$ we set
\begin{eqnarray}
 (i)\,\,\,\Lambda_x(s)&=&\gamma\, s^{2\sigma}||x||_{l^{2(\sigma+1)}}^{2(\sigma+1)}\,\,\,\,{\rm if}\,\,V_0<0.\nonumber\\
(ii)\,\,\, \Lambda_x(s)&=&\gamma\, s^{2\sigma}||x||_{l^{2(\sigma+1)}}^{2(\sigma+1)}+V_0x_0^2\,\,\,\,{\rm if}\,\,V_0>0.\nonumber
\end{eqnarray}
For $\gamma>0$, $\Lambda_x(s)$ is strictly monotonically growing for $0<s<\infty$. As noted in Section \ref{section:breathers}, $sx\in \mathcal{N}$ if and only if
\begin{equation}
 \Lambda_x(s)-||\nabla x||_{l^2}^2+\omega ||x||_{l^2}^{2}+V_0x_0^{2}=0.\label{eq:Leq}
\end{equation}
(i) As $\omega<0$- and $V_0<0$, we infer that there exists a unique $s(x)>0$ for which (\ref{eq:Leq}) holds. Consequently,
\begin{equation}
 \Gamma_x^\prime(s(x))=s(x)\left(\Lambda(s(x))-||\nabla x||_{l^2}^2+\omega ||x||_{l^2}^{2}+V_0x_0\right)=0,\nonumber
\end{equation}
that is, $s(x)x\in \mathcal{N}$. From Section we know that $\Gamma(s)$ is strictly increasing  for $0<s<s(x)$ and strictly decreasing for $s(x)<s<\infty$.
The case (ii) is treated analogously.
Hence, claim  (\ref{eq:claim1}) is valid.

\vspace*{0.5cm}

\underline{Claim 2:} For every $\nu>0$ and $x\in l^2$, it holds,
\begin{equation}
 J(x)\ge \underline{E}_\nu-\frac{1}{2}\omega \nu.\label{eq:claim2}
\end{equation}
The equality is true if and only if $x$ is a minimizer of {\bf M1}, and $x$ is a ground state solution of {\bf M2}.

\vspace*{0.3cm}

Indeed,
as $x\in \mathcal{N}$, from the previous step  we have,
\begin{equation}
 J(x)= \max_{s>0}J(sx)\ge J(sx),
\end{equation}
and $J(x)=J(sx)$ if and only if $s=1$.
Further, we have
\begin{equation}
 J(x)\ge J\left(\frac{\nu^{1/2}}{||x||_{l^2}}\,x \right)=E\left(\frac{\nu^{1/2}}{||x||_{l^2}}\,x \right)-\frac{1}{2}\omega \nu\ge \underline{E}_\nu-\frac{1}{2}\omega \nu.\label{eq:Jineq}
\end{equation}
Hence,  relation (\ref{eq:claim2}) is true. If, moreover, the equality is true, then due to (\ref{eq:Jineq}), it holds
\begin{equation}
 E\left(\frac{\nu^{1/2}}{||x||_{l^2}}\,x \right)=\underline{E}_\nu,
\end{equation}
and
\begin{equation}
 J(x)=J\left(\frac{\nu^{1/2}}{||x||_{l^2}}\,x \right).
\end{equation}
By claim 1, this implies $||x||_{l^2}^2=\nu$, and hence, $E(x)=\underline{E}_\nu$.
That is, $x$ is a minimizer of {\bf M1}.
By (\ref{eq:claim2}), for every $x\in \mathcal{N}$, we also have,
\begin{equation}
 J(x)\ge \underline{E}_\nu-\frac{1}{2}\omega \nu=E(x)-\frac{1}{2}\omega ||x||_{l^2}^2=J(x),
\end{equation}
implying that $x$ is a ground state of {\bf M2}. Moreover, if $x$ is a minimizer of {\bf M1}, then
\begin{equation}
 J(x)=E-\frac{1}{2}\omega ||x||_{l^2}^2=\underline{E}_\nu-\frac{1}{2}\omega ||x||_{l^2}^2,
\end{equation}
so that the equality holds and claim 2 holds true.

For the final step of the proof,  for given
$\omega \in \{\omega(y)\,:\,y\in S_\nu\,\,{\text is\,\,a\,\,minimizer\,\,of\,\,{\bf M1}}\}$, let $y\in l^2$ be any ground state solution of {\bf M2}. Then,
\begin{equation}
 J(y)\le \underline{E}_\nu-\frac{1}{2}\omega\mu.
\end{equation}
giving in conjunction with (\ref{eq:claim2}) $J(y)=\underline{E}_\nu-(\omega/2) \nu$. Finally,
by claim 2 it follows that $y$ is a minimizer of {\bf M1} finishing the proof.
\end{proof}

%

 \begin{lemma}\label{lemma:Eminus}
 For every $\nu>0$ holds
 \begin{equation}
  \inf_{x\in S_\nu}E(x)=\underline{E}_\nu<0.\nonumber
 \end{equation}
 \end{lemma}

 \begin{proof}
 For $x \in S_\nu$ we take $sx$ with $s>0$ and
 \begin{eqnarray}
  E(sx)&=&\frac{1}{2}\sum_{n}\left((sx_{n+1}-sx_n)^2-\frac{\gamma}
  {\sigma+1} (sx_n)^{2(\sigma+1)}-V_0(sx_0)^2\right)\nonumber\\
   &=&\frac{s^2}{2}\sum_{n}\left((x_{n+1}-x_n)^2-\frac{s^{2\sigma}}
   {\sigma+1} x_n^{2(\sigma+1)}-V_0 x_0^2\right)\nonumber\\
  &=&s^2 M_{x,1}-s^{2(\sigma+1)}M_{x,2}:=f(s).\nonumber
 \end{eqnarray}
 As for $s$ sufficiently large, the term $s^{2(\sigma+1)}$ dominates it follows $E<0$. If $\inf_{x\in S_\nu}E(x)=\underline{E}_\nu$ exists, then $\underline{E}_\nu<0$.
 Furthermore, we bound $E$ from below as follows
 \begin{eqnarray}
 E&\ge& -\frac{\gamma}{2(\sigma+1)}||x||_{l^{2\sigma+2}}^{2\sigma+2}-\frac{V_0}{2}||x||_{l^2}^2 \ge -\frac{\gamma}{2(\sigma+1)}||x||_{l^{2}}^{2\sigma+2}-\frac{V_0}{2}||x||_{l^2}^2\nonumber\\
 &\ge& -\left(\frac{\gamma}{2(\sigma+1)}\nu^{\sigma}+\frac{V_0}{2}\right)\nu>-\infty.\nonumber
 \end{eqnarray}
Since $E$ is bounded below, it attains its infimum $\underline{E}_\nu<0$ in $S_\nu$.
 \end{proof}

\vspace*{0.5cm}

Since it is proven that a ground state of problem {\bf M1} is also a ground state of problem {\bf M2}, and due to the energy equivalence  in relation (\ref{eq:varprob}), we apply
Weinstein's criterion  \cite{Weinstein}, valid for  the solution of problem  {\bf M2}, to investigate the existence of excitation thresholds for the formation of ground states. In detail, if a ground state (as a minimizer of $\inf_{x\in S_\nu}E(x)<0$) exists for any value of $\nu>0$, there is no excitation threshold. Conversely, if there exists a positive $\nu_{exc}$ such that $\inf_{x\in S_\nu}E(x)<0$ if and only if $\nu>\nu_{exc}$, then $\nu_{exc}$ constitutes an excitation threshold.

In what follows,  we assume that $x$ is represented by exponentially decaying square-summable sequences for the following reasons:

1) As seen in  Section \ref{section:linear}, in the linear limiting case, $\gamma=0$, the defect modes are exponentially localized at a single site (without loss of generality the origin  of the lattice). Moreover, these defect modes are preserved for (at least) weak nonlinearity.

2) With concern to the (exponential) localization of ground states of the nonlinear system with $\gamma > 0$, we exploit that for a self-adjoint operator $A$ on a Hilbert space $X$ and a compact operator $K$ on $X$, one has that  $\sigma(A+K)\setminus \sigma(A)$ consists of isolated eigenvalues of finite multiplicity \cite{Douglas}.
We present a consice proof of exponential localization of the solutions:
\begin{proposition}\label{proposition:expdecay}
  Solutions $x\in l^2(\mathbb{Z}^d)$ to
\begin{equation}
\omega x_n+(\Delta x)_n+\gamma x_n^{2\sigma} x_n=0,\,\,\,n\in \mathbb{Z}^d,\label{eq:breathers}
\end{equation}
 satisfy
 \begin{equation}
  |x_n|\le C\exp(-\eta |n|),\,\,\,n\in \mathbb{Z}^d,\nonumber
 \end{equation}
 with some constants $C,\eta>0$.
\end{proposition}

\begin{proof}
We express (\ref{eq:breathers}) in the form
\begin{equation}
 -(\Delta+M)x_n=\omega x_n,\label{eq:eigenvalue}
\end{equation}
where the multiplication operator $M\,:\,l^2\mapsto l^2$ is determined by
\begin{equation}
Mx_n =x_n^{2\sigma}\cdot x_n.\nonumber
\end{equation}
Since $\lim_{|n|\rightarrow \infty}x_n^{2\sigma}=0$, the multiplication operator $M$ is compact on $l^2$.
As $\omega \notin \sigma(\Delta)$ implies $\omega \in \sigma(\Delta +M)\setminus \sigma(\Delta)$, it follows that $\omega$ is an eigenvalue of finite multiplicity of the operator $\Delta+M$. The associated eigenfunction exhibits exponential decay (see e.g. in \cite{Jacobi}) and the proof is finished.
\end{proof}

%
%
%

\vspace*{0.5cm}

 With the proven exponential localization of ground states, 
 $(x_n)_{n\in \mathbb{Z}^d} \in l^2$ can be represented as    $(x_n)_{n\in \mathbb{Z}^d}=(A\exp(-\alpha |n|))_{n\in \mathbb{Z}^d}$ with $\alpha>\beta$. (We recall that in systems with dispersive interaction $\Delta x$ their solutions cannot decay faster than exponential.)
 
Using $\alpha=-\log(\eta)$ with $0<\eta <1$, we express $x_n=A\exp(-\alpha |n|)$ as 
 \begin{equation}
  x_n=A\eta^{|n|},\,\,\,n \in \mathbb{Z}^d,\label{eq:trial}
 \end{equation}
 with $|n|=|n_1|+...+|n_d|$ and amplitude $A>0$ determined by 
 \begin{equation}
  A=\left(\left(\frac{1-\eta^2}{1+\eta^2}\right)^d\,\nu \right)^{1/2}.\nonumber
 \end{equation}
Note that $\eta \rightarrow 0^+$ leads to enhanced localization towards a single-site state  while $\eta \rightarrow 1^-$ renders the profile wider (advancing towards  a flat state of zero amplitude), i.e. varying $\eta$ between zero and one interpolates between two nearly extreme states, namely a single site state and a wide (yet weakly localized)  absolutely continuous spectrum near edge state.

The energy $E(x)$ expressed in terms of
\begin{equation}
 x_n=\left(\left(\frac{1-\eta^2}{1+\eta^2}\right)^d\,\nu \right)^{1/2}\eta ^{|n|},\nonumber
\end{equation}
reads
\begin{equation}
 E=\left(\frac{(1+\eta^2)^{d}-(2\eta)^d}{(1+\eta^{2})^d}-\frac{\gamma}{\sigma+1} \left(\left(\frac{1-\eta^2}{1+\eta^2}\right)^d \right)^{\sigma+1} \left(\frac{1+\eta^{2\sigma+2}}{(1-\eta)F(\eta,\sigma)}\right)^d\nu ^\sigma-V_0\left(\frac{1-\eta^2}{1+\eta^2}\right)^d\right)\,\nu,\label{eq:Hnu}
\end{equation}
where $F(\eta,\sigma)=\sum_{l=1}^{2\sigma+1}\eta^l+1$.

Subsequently we
 investigate  for which value(s) of the mass $\nu=||x||_{l^2}^2$ a  ground state solution in the form of a  breather respectively impurity mode exists at all.

 In this context we use the  criterion
  $\underline{E}_{\nu}\le E<0$ that is necessary and sufficient for a state to be localized. The energy $E=E(\nu;\eta,\gamma,\sigma,V_0)$ is determined by (\ref{eq:Hnu}).

We begin  with the effect of the point defect on the existence of excitation thresholds 
for the formation of ground state breather solutions for $\gamma>0$.

\vspace*{0.5cm}

\noindent{\it Proof of Theorem \ref{theorem:threshold}}

\noindent 1) Attractive delta potential $V_0>0$.

As noted in section \ref{section:linear},  for $V_0>0$  the linear lattice with delta potential possesses a bound state (defect mode) of negative energy generated 
by the presence of the point defect. 

With the addition of a focusing nonlinear term to the linear equation, its associated contribution to the energy is
negative so that the total energy remains negative for all $\nu>0$. Hence, for  focusing nonlinearity $\gamma>0$ and attractive delta potential $V_0>0$ the existence of an excitation threshold for the creation of breathers is ruled out.

\vspace*{0.5cm}

\noindent 2) Repulsive delta potential $V_0<0$.

From (\ref{eq:Hnu}) we deduce that $E\ge0$ if and only if
the right-hand side  of  (\ref{eq:Hnu}) is positive yielding the condition
 \begin{eqnarray}
\frac{\sigma+1}{\gamma} \left[2\frac{(1+\eta^2)^{d}-(2\eta)^d}{(1+\eta^{2})^d}-V_0\left(\frac{1-\eta^2}{1+\eta^2}\right)^d\right] \left(\left(\frac{1+\eta^2}{1-\eta^2}\right)^d \right)^{\sigma+1} \left(\frac{F(\eta,\sigma)(1-\eta)}{1+\eta^{2\sigma+2}}\right)^d&\ge&\nu^{\sigma}.
\nonumber
 \end{eqnarray}
This  implies  an upper bound of  $\nu$ as follows
\begin{eqnarray}
&&\frac{\sigma+1}{\gamma}\inf_{0<\eta <1}\left(\frac{1}{(1+\eta^2)^d}\left[2\left((1+\eta^2)^{d}-(2\eta)^d \right)- V_0(1-\eta^2)^d\right]
\frac{(1+\eta^2)^{\sigma d}}{(1-\eta^2)^{(\sigma+1)d}}\right.\nonumber\\
&\times&\left.\left(\frac{F(\eta,\sigma)(1-\eta)}{1+\eta^{2\sigma+2}}\right)^d\right)\ge \sup_{\nu\ge 0}\nu^{\sigma}.
\label{eq:inf1}
\end{eqnarray}
The inequality (\ref{eq:inf1}) is of the form
\begin{eqnarray}
 \nu^{\sigma}&\le& \frac{\sigma+1}{\gamma}\inf_{0<\eta <1 }\left(\frac{1}{(1+\eta^2)^d}\left[2F_1(\eta,d)(1-\eta)^2 -V_0(1+\eta)^d(1-\eta)^d\right]\right.\nonumber\\
 &\times&\left.
\frac{(1+\eta^2)^{\sigma d}}{(1+\eta)^{(\sigma+1)d}(1-\eta)^{(\sigma+1)d}}\left(\frac{F(\eta,\sigma)(1-\eta)}{1+\eta^{2\sigma+2}}\right)^d\right),\label{eq:psigma}
 \end{eqnarray}
where $F_1(\eta,d)(1-\eta)^2=(1+\eta^2)^{d}-(2\eta)^d$ is  explicitly given for $d=1,...,4$ by   
$F_1(\eta,1)=1$, $F_1(\eta,2)=(1+\eta)^2$, $F_1(\eta,3)=\eta^4+2\eta^3+6\eta^2+2\eta+1$, $F_1(\eta,4)=\eta^6+2\eta^5+7\eta^4+12\eta^3+7\eta^2+2\eta+1$.

\vspace*{0.5cm}

(i) $\sigma <\min\{1,2/d\}$.
The infimum on the right-hand side of (\ref{eq:psigma}) is attained for $\eta \rightarrow 1^-$.
 Setting $1-\eta=\epsilon$,  we obtain
\begin{eqnarray}
 \nu^{\sigma}&\le& \frac{\sigma+1}{\gamma}\lim_{\epsilon \rightarrow 0^+}\left(\frac{(1+\eta^2)^{\sigma d}}{(1+\eta)^{(\sigma+1)d}}\frac{1}{(1+\eta^2)^d}\left(\frac{F(\eta,\sigma)}{1+\eta^{2\sigma+2}}\right)^d\right)\nonumber\\
 &\times&\left[2F_1(\eta,d)\epsilon^{2-\sigma d} -V_0(1+\eta)^d \epsilon^{d(1-\sigma)}\right]=0.\nonumber
\end{eqnarray}
We conclude, only for $\nu=0$ it holds $E\ge 0$. That is, for $\sigma <\min\{1,2/d\}$
one has $\underline{E}_\nu\le E<0$ for all $\nu>0$ so that  there is no excitation threshold. Hence, ground states  of arbitrarily small amplitude that bifurcate off the linear zero solution exist.

\vspace*{0.5cm}

(ii) $\sigma\ge \max\{1,2/d\}$. The infimum  on the right-hand side of (\ref{eq:psigma}) is attained for $\eta \rightarrow 0^+$
yielding
\begin{eqnarray}
 \nu^{\sigma}&\le& \frac{\sigma+1}{\gamma}\lim_{\eta \rightarrow 0^+}\left(\frac{1}{(1+\eta^2)^d}\left[2F_1(\eta,d)(1-\eta)^2 -V_0(1+\eta)^d(1-\eta)^d\right]\right.\nonumber\\
&\times&\left.\frac{(1+\eta^2)^{\sigma d}}{(1+\eta)^{(\sigma+1)d}(1-\eta)^{(\sigma+1)d}}\left(\frac{F(\eta,\sigma)(1-\eta)}{1+\eta^{2\sigma+2}}\right)^d\right)=\frac{\sigma+1}{\gamma}\left(2-V_0\right).\nonumber
\end{eqnarray}
Thus, $E\ge 0$ if and only if
\begin{equation}
0\le\nu \le \nu_{1}=\left(\frac{(\sigma+1)(2-V_0)}{\gamma}\right)^{1/\sigma}.
\end{equation}

Conversely,  $\underline{E}_\nu \le E<0$ if and only if
\begin{equation}
\nu > \nu_{1}=\left(\frac{(\sigma+1)(2-V_0)}{\gamma}\right)^{1/\sigma}>0.\label{eq:uppernu}
\end{equation}
Notice that condition (\ref{eq:uppernu}) is the condition that   a state becomes localized   at all because $\nu_1$ gives the threshold value that the mass $\nu$ of a state  ($\eta \sim 0^+$) must exceed in order to be (at least) weakly localized ($\eta \sim 0^+$)  with its   energy $E(\eta\sim 0^+)$  (just) below zero.

Writing (\ref{eq:Hnu}) as
\begin{equation}
 E(\nu;\eta,d,\sigma)=\left(A_{\eta,d} -B_{\eta,d,\sigma} \nu^\sigma\right)\nu,\nonumber
\end{equation}
where
\begin{eqnarray}
 A_{\eta,d}&=&2\frac{(1+\eta^2)^{d}-(2\eta)^d}{(1+\eta^{2})^d}-V_0\left(\frac{1-\eta^2}{1+\eta^2}\right)^d>0,\nonumber\\
 B_{\eta,d,\sigma}&=&\frac{\gamma}{\sigma+1} \left(\left(\frac{1-\eta^2}{1+\eta^2}\right)^d \right)^{\sigma+1} \left(\frac{1+\eta^{2\sigma+2}}{(1-\eta)F(\eta,\sigma)}\right)^d>0,\nonumber
\end{eqnarray}
one notices that $\nu_1$ is the unique nontrivial zero of $E(\nu;0^+,d,\sigma)$. That is, $E(\nu;0^+,d,\sigma)\ge 0$ for $\nu\le \nu_1$ and $E(\nu;0^+,d,\sigma)<0$ for $\nu>\nu_1$. Thus, $\nu_1=\nu_{exc}$.

In conclusion, $\underline{H}_\nu=2\underline{E}_\nu<0$ if and only if $\nu > \nu_{exc}>0$, i.e. there is an excitation threshold for the existence of ground states finishing the proof of Theorem \ref{theorem:threshold}.

\vspace*{0.5cm}

To gain further insight into the localization features of the DNLS (\ref{eq:system2}), we note that  when $\eta \rightarrow 0^+$ the excitation pattern approaches the single-site
state and   the  gap between the energy of the linear system, 
\begin{equation}
E_{0}=\frac{(1-\eta)^2}{1+\eta^2}\nu,\nonumber
\end{equation}
and that of the nonlinear contribution (responsible for creating the potential well which makes possible at all localization), 
\begin{equation}
E_1=- \frac{\gamma}{2\sigma+2} \left(\frac{1-\eta^2}{1+\eta^2}\,\nu \right)^{\sigma+1} \frac{1+\eta^{2\sigma+2}}{1-\eta^{2\sigma+2}},\nonumber
\end{equation}
becomes maximal. In order to ease the   illustration we treat here the case $d=1$. 
We obtain
\begin{equation}
 \lim_{\eta \rightarrow 0^+}\left(E_0(\eta)+E_1(\eta)\right)=\nu\left(1-\frac{\gamma}{2\sigma+2}\nu^\sigma\right)\ge 0 \Longleftrightarrow \nu\le \left(\frac{2\sigma+2}{\gamma}\right)^{1/\sigma}
\end{equation}
Substituting the threshold value $\nu_{exc}= (2\sigma+2)/\gamma)^{1/\sigma}$ into the expression for the energy $E=E_0+E_1$, it is readily seen that for $\sigma \ge 2$
\begin{equation}
 E=\left(1-\frac{(1+\eta)^{\sigma+1}(1+\eta^{2\sigma+2}}{(1+\eta^2)^{\sigma+1}}   (1-\eta)^{\sigma-1}\right)\frac{(1-\eta)^2}{1+\eta^2}\nu\ge 0,\qquad \forall \eta \in (0,1).\nonumber
\end{equation}
Conclusively, $E<0$ if $\nu>\nu_{exc}=(2\sigma+2)/\gamma)^{1/\sigma}$.\nonumber

\  \  $\Box$

\vspace*{0.5cm}

We also discuss the localization behavior when $\nu\rightarrow \infty$.
From the relation
\begin{equation}
 (1+\eta^2)^d A^2=(1-\eta^2)^d\nu,\nonumber
\end{equation}
we deduce that if $\eta \rightarrow 1^-$ then $A\rightarrow 0^+$ complying with the fact that low amplitude breathers bifurcate off  the linear trivial solution. Conversely, if $\eta\rightarrow 0^+$ then  $A^2\rightarrow \nu$,  
in accordance with the result in \cite{Weinstein} proving that as $\nu$ increases, the ground states grow in amplitude and become increasingly concentrated about a single lattice site. 

\vspace*{0.5cm}

We conclude this section with an  investigation of  the persistence of the linear
impurity mode in the presence of defocusing nonlinearity,  i.e. $\gamma<0$ and $V_0\neq 0$. 

\vspace*{0.5cm}

\noindent 1) Defocusing nonlinearity $\gamma<0$ and attractive delta potential $V_0>0$.

\begin{equation}
 E=2\nu\frac{(1+\eta^2)^{d}-(2\eta)^d}{(1+\eta^{2})^d}-\frac{\gamma}{\sigma+1} \left(\left(\frac{1-\eta^2}{1+\eta^2}\right)^d\,\nu \right)^{\sigma+1} \left(\frac{1+\eta^{2\sigma+2}}{(1-\eta)F(\eta,\sigma)}\right)^d-V_0\left(\frac{1-\eta^2}{1+\eta^2}\right)^d\,\nu,\nonumber
\end{equation}

$E\le 0$  if
\begin{eqnarray}
\nu^{\sigma}&\le&\frac{\sigma+1}{\gamma}\inf_{0<\eta <1}\,\left(V_0\left(\frac{1+\eta}{1+\eta^2}\right)^d (1-\eta)^{d(1-\sigma)}- 2\frac{F_1(\eta,d)}{1+\eta^{2}}(1-\eta)^{2-d\sigma}\right)\nonumber\\
&\times&\left(1+\eta^2\right)^{d(\sigma+1)} \left(\frac{F(\eta,\sigma)}{1+\eta^{2\sigma+2}}\right)^d.
\label{eq:infgamma}
\end{eqnarray}
If $\sigma<\min\{1,2/d\}$, then the infimum on the right-hand side of (\ref{eq:infgamma})
is attained for $\eta\rightarrow 1^-$ yielding $\nu\le 0$, 
so that we deduce $E> 0$ for all  $\nu> 0$. That is, for $\sigma <1$ a localized state cannot exist.
However, if $\sigma \ge \max\{1,2/d\}$,  the 
  infimum on the right-hand side of (\ref{eq:infgamma})   is attained for 
  $\eta \rightarrow 0^+$ yielding the {\it upper threshold} 
\begin{equation}
0<\nu< \overline{\nu}_{exc}\le \frac{\sigma+1}{\gamma} \left(V_0-2\right),\,\,\,V_0>2.\nonumber
\end{equation}
  Hence, the (linear) impurity mode only sustains defocusing nonlinearity 
  for subcritical mass $\nu<\overline{\nu}_{exc}$.

\vspace*{0.5cm}  
  \setcounter{equation}{9}
\noindent 2) Focusing nonlinearity $\gamma>0$ and repulsive delta potential $V_0<0$.

The existence of the staggering defect mode requires $E>4d$, which entails the inequality
\begin{eqnarray}
\nu^{\sigma+1}&<&\frac{\sigma+1}{\gamma}\inf_{-1<\eta <0}\,\left(2\frac{F_1(\eta,d)}{1+\eta^{2}}(1-\eta)^{2-d\sigma}\nu-V_0\left(\frac{1+\eta}{1+\eta^2}\right)^d (1-\eta)^{d(1-\sigma)}\nu\right.\nonumber\\
&-&\left. 4d(1-\eta)^{-d\sigma}\right)\left(1+\eta^2\right)^{d(\sigma+1)}\left(\frac{F(\eta,\sigma)}{1+\eta^{2\sigma+2}}\right)^d\label{eq:survstagg}
\end{eqnarray}
must be satisfied so  the 
  staggering defect mode survives under the impact of focusing nonlinearity. For $\sigma \ge \max\{1,2/d\}$ the infimum on the right-hand side of (\ref{eq:survstagg}) is attained for $\eta \rightarrow 0^-$  resulting in
\begin{equation}
 0<\nu^{\sigma+1}<\frac{\sigma+1}{\gamma}\left((2-V_0)\nu-4d\right).\nonumber
\end{equation}
  In contrast, 
  for $\sigma < \min\{1,2/d\}$ one has $E\le 4d$ for all $\nu^{\sigma+1}\ge ((2-V_0)\nu-4d)(\sigma+1)/\gamma$.
  Thus, then no staggering defect mode exists.
  
\vspace*{0.5cm}

\begin{remark}\label{remark:simultaneous}
For $\gamma>0$ and $V_0<0$ the simultaneous existence of localized states below (breathers) as well as above (staggering defect mode) the absolutely continuous spectrum of $\Delta$ is possible. 
This is an effect of discreteness and it is not present in the continuum NLS with point defect because the operator $-\partial^2_x+V_0\delta(x)$, $V_0>0$, possesses a purely absolutely continuous spectrum  equal to $[0,\infty)$ and no eigenvalue (point spectrum). 
\end{remark}

 \vspace*{0.5cm}

\section{Asymptotic behavior}\label{section:scattering}
In this section we discuss the asymptotic features of the $\delta$DNLS 
\begin{equation}
i\frac{d u_n}{dt}+ (\Delta_\delta u)_n+\gamma|u_n|^{2\sigma}u_n=0,\,\,\,n\in \mathbb{Z}^d.\nonumber
\end{equation}

\vspace*{0.5cm}
We begin with the case $V_0=0$.

{\it Scattering:}
Concerning the asymptotic properties
we consider the linear Schr\"odinger group $\exp(i t \Delta)$.
In the regime where no nonlinear localized state exists, that is for  $V_0<0$, $\sigma\ge \max\{1,2/d\}$ and $V_0=0$, $\sigma \ge 2/d$, respectively, and $l^2$ norm below $\nu_{exc}$, we establish that the solutions
 scatter to a  solution of the linear problem in $l^p$, $p>2$. That is,
the solutions of the nonlinear problem exhibit asymptotically free behavior.
A solution is said to scatter in the positive (negative) time direction if there exists $v^{\pm }  \in l^p$
such that
\begin{equation}
 \left\|u(t)-\exp(i\Delta t)v^{\pm }  \right\|_{l^p}\longrightarrow 0\,\,\,{\rm as}\,\,\,t\rightarrow \pm \infty.\nonumber
\end{equation}
The solution scatters if it scatters in both time directions.

\vspace*{0.5cm}

\noindent {\it Proof of Theorem \ref{theorem:scatterin1}}
 We begin with the treatment of scattering in the positive time direction. 
 For the Cauchy problem  with initial datum $u_0\in l^2$ we consider the 
 global solution $u(t)$ and introduce the asymptotic state $v^+$
 \begin{equation}
  v^+=u_0+i\gamma \int_0^\infty\,\exp(-i\Delta s) |u(s)|^{2\sigma}u(s) ds.\label{eq:int1}
 \end{equation}
Application of the operator $\exp(i\Delta t)$ on either side of
(\ref{eq:int1}) yields
\begin{eqnarray}
 \exp(i\Delta t) v^+&=& \exp(i\Delta t)u_0+i \gamma  \int_0^\infty\,\exp(i\Delta(t- s))|u(s)|^{2\sigma}u(s)ds\nonumber\\
 &=& u(t)+i \gamma \int_t^\infty\,\exp(i\Delta (t- s))|u(s)|^{2\sigma}u(s) ds.\nonumber
\end{eqnarray}
Using the  estimate (5.2) 
we derive
\begin{eqnarray}
 \left\|\exp(i\Delta t) v^+-u(t) \right\|_{l^p}&=& \gamma \left\|\int_t^\infty\,\exp(i\Delta (t- s))  |u(s)|^{2\sigma}u(s)ds \right\|_{l^p}\nonumber\\
 &\le&  \gamma\int_0^\infty\,\left\|\exp(i\Delta (t- s))  |u(s)|^{2\sigma}u(s)\right\|_{l^{p}}ds\nonumber\\
 &\le&  C\int_0^\infty\,\frac{1}{<t-s>^{d(p-2)/(3p)}}
\gamma  \left\| |u(s)|\right\|_{l^{(2\sigma+1)p^\prime}}^{2\sigma +1}ds,\nonumber
\end{eqnarray}
and we used the Strichartz estimate \cite{Stefanov}
\begin {equation}
||\exp(i \Delta t )u(0)||_{l^p}\le C<t>^{-d(p-2)/(3p)}||u(0)||_{l^{p^\prime}},\qquad 2\le p\le \infty,\nonumber
\end {equation}
arising from an interpolation between
\begin{equation}
 ||u(t)||_{l^2}=||u(0)||_{l^2},\nonumber
\end{equation}
and
\begin{equation}
 ||u(t)||_{l^\infty}=C<t>^{-d/3}||u(0)||_{l^1},\nonumber
\end{equation}
where $<t>=1+|t|$.
By assumption $2(\sigma+1)\ge p$ one has $(2\sigma+1)p^\prime \ge p$. Then due to the continuous embeddings $l^{(2\sigma+1)p^\prime} \subseteq l^p$ we have
\begin{eqnarray}
\left\|\exp(i\Delta t) v^+-u(t) \right\|_{l^p}&\le&  C  \int_0^\infty\,\frac{1}{<t-s>^{d(p-2)/(3p)}}\,
\gamma \left\| u(s)\right\|_{l^{p}}^{2\sigma+1}ds\nonumber\\
&\le&C_1  \int_0^\infty\,\frac{1}{<t-s>^{d(p-2)/(3p)}}\,
\frac{\gamma \left\| u(0)\right\|_{l^{p^\prime}}^{2\sigma+1}}{s^{d(p-2)(2\sigma+1)/(3p)}}ds\nonumber\\
&\le&C_1 \gamma \left\| u(0)\right\|_{l^{p^\prime}}^{2\sigma+1} \int_0^\infty\,\frac{1}{<t-s>^{d(p-2)/(3p)}}\,
\frac{1}{<s>^{d(p-2)(2\sigma+1)/(3p)}}ds.\nonumber
\end{eqnarray}
We evaluate the integral as follows
\begin{eqnarray}
&& \int_0^\infty\,\frac{1}{<t-s>^{d(p-2)/(3p)}}\,
\frac{1}{<s>^{d(p-2)(2\sigma+1)/(3p)}}ds\nonumber\\
&=&\int_0^t\,\frac{1}{<t-s>^{d(p-2)/(3p)}}\,
\frac{1}{<s>^{d(p-2)(2\sigma+1)/(3p)}}ds\nonumber\\
&+&\int_t^\infty\,\frac{1}{<t-s>^{d(p-2)/(3p)}}\,
\frac{1}{<s>^{d(p-2)(2\sigma+1)/(3p)}}ds\nonumber\\
&\equiv& I_1+I_2.\nonumber
\end{eqnarray}
Our assumption $p\ge 4d\sigma/(2d\sigma-3)$ implies $d(p-2)(2\sigma+1)/(3p)>1$. Therefore, for $I_1$ we estimate  
\begin{equation}
I_1=\int_0^t\,\frac{1}{<t-s>^{d(p-2)/(3p)}}\,
\frac{1}{<s>^{d(p-2)(2\sigma+1)/(3p)}}ds\le C\frac{1}{<t>^{d(p-2)/(3p)}}.\nonumber
\end{equation}
For $I_2$ we proceed as follows
\begin{eqnarray}
 I_2&=&\int_t^\infty\,\frac{1}{<t-s>^{d(p-2)/(3p)}}\,
\frac{1}{<s>^{d(p-2)(2\sigma+1)/(3p)}}ds\nonumber\\
&=&\int_t^\infty\,\frac{1}{(1+s-t)^{d(p-2)/(3p)}}\,
\frac{1}{(1+s)^{d(p-2)(2\sigma+1)/(3p)}}ds.\nonumber
\end{eqnarray}
The change of variables $s=t z$ gives
\begin{eqnarray}
 I_2
&=&\int_1^\infty\,\frac{t}{(1+t(z-1))^{d(p-2)/(3p)}}\,
\frac{1}{(1+tz)^{d(p-2)(2\sigma+1)/(3p)}}dz.\nonumber
\end{eqnarray}
A further change of variables $z=1/s$ results in
\begin{eqnarray}
 I_2
&=&\int_0^1\,\frac{t}{s^2}\frac{1}{(1+ts^{-1}(1-s))^{d(p-2)/(3p)}}\,
\frac{1}{(1+ts^{-1})^{d(p-2)(2\sigma+1)/(3p)}}ds\nonumber\\
&\le& \frac{1}{t^{2(\sigma+1)d(p-2)/(3p)-1}}\int_0^1\,\left(1-s\right)^{-d(p-2)/(3p)}\,s^{2(\sigma+1)d(p-2)/(3p)-2}ds\nonumber\\
&=& \frac{1}{t^{2(\sigma+1)d(p-2)/(3p)-1}}\,B\left(2(\sigma+1)d(p-2)/(3p)-1,1-d(p-2)/(3p)\right),\nonumber
\end{eqnarray}
where $B(x,y)=\int_0^1 s^{x-1}(1-s)^{y-1}ds$ defines the beta-function.
We obtain
\begin{eqnarray}
 \left\|\exp(i\Delta t) v^+-u(t) \right\|_{l^p}&\le& C\left(\frac{1}{t^{d(p-2)/(3p)}}\right.\nonumber\\
 &+&\frac{1}{t^{2(\sigma+1)d(p-2)/(3p)-1}}\nonumber\\
 &\times&\left.B\left(2(\sigma+1)d(p-2)/(3p)-1,1-d(p-2)/(3p)\right)\right),\,\,\,\,\forall t>0,\nonumber
\end{eqnarray}
so that we deduce $v^+ \in l^p$.

Furthermore, applying the operator $\exp(-i\Delta t)$
on both sides of the integral equation
\begin{equation}
 u(t)=\exp(i\Delta t)u_0-i\gamma \int_0^t\,\exp(i\Delta (t- s))  |u(s)|^{2\sigma}u(s)ds, \nonumber
\end{equation}
we obtain 
\begin{eqnarray}
 \exp(-i\Delta t) u(t)&=& u_0-i \gamma\int_0^t\,\exp(-i\Delta s))  |u(s)|^{2\sigma}u(s)ds\nonumber\\
 &=& v^+ +i\gamma\int_t^\infty\,\exp(-i\Delta s)  |u(s)|^{2\sigma}u(s)ds.\nonumber
\end{eqnarray}
Thus,
\begin{equation}
 \exp(-i\Delta t) u(t)-v^+=i \int_t^\infty\,\exp(-i\Delta s) \gamma |u(s)|^{2\sigma}u(s)ds,\nonumber
\end{equation}
and we get 
\begin{eqnarray}
 \left\|\exp(-i\Delta t)u(t) -v^+\right\|_{l^p}&=& \gamma\left\| \int_t^\infty\,\exp(-i\Delta s)  |u(s)|^{2\sigma}u(s)ds \right\|_{l^p}\nonumber\\
 &=& \gamma\left\| \int_t^\infty\,\exp(-i\Delta t)\exp(i\Delta (t-s))  |u(s)|^{2\sigma}u(s)ds \right\|_{l^p}\nonumber\\
 &\le& \gamma  \int_0^\infty\,\left\|\exp(-i\Delta t)\exp(i\Delta (t-s)) |u(s)|^{2\sigma}u(s)\right\|_{l^p}ds \nonumber\\
&\le&  C \gamma 
 \int_0^\infty\,\frac{\left\| u(s)\right\|_{l^{(2\sigma +1)p^\prime}}^{2\sigma +1}}{<t-s>^{d(p-2)/(3p)}} ds\,\left\| \exp(-i\Delta t)\right\|_{\mathcal{L}(l^p,l^{p^\prime})}\nonumber\\
 &\le&  C_1 \gamma \,\frac{\left\| u(0)\right\|_{l^{p^\prime}}^{2\sigma +1}}{<t>^{d(p-2)/(3p)}}
 \int_0^\infty\,\frac{1}{<t-s>^{d(p-2)/(3p)}}\,\frac{1}{<s>^{d(p-2)(2\sigma+1)/(3p)}}ds\nonumber\\
 &\le&  C_2 \gamma \,\left\| u(0)\right\|_{l^{p^\prime}}^{2\sigma +1} \,\frac{1}{<t>^{d(p-2)/(3p)}}\left(\frac{1}{<t>^{d(p-2)/(3p)}}\right.\nonumber\\
 &+&\left.\frac{1}{t^{2(\sigma+1)d(p-2)/(3p)-1}}\,B\left(2(\sigma+1)d(p-2)/(3p)-1,1-d(p-2)/(3p) \right)\right), \nonumber
\end{eqnarray}
for all $t>0$. Since by assumption $2(\sigma+1)(p-2)/(3p)-1>d(p-2)/(3p)>0$  it follows
\setcounter{equation}{4}
\begin{equation}
\left\|\exp(-i\Delta t)u(t) -v^+\right\|_{l^p}
\le
C||u_0||_{l^{p^\prime}}\frac{1}{t^{d(p-2)/(3p)}},\,\,\,t>0.\nonumber
\end{equation}

 Scattering in the negative time direction can be dealt with in the same way as above in  
 the positive time direction. In fact, due to the time reversibility of system (\ref{eq:system1}) it holds $v^-=\overline{v}^+$.   
Hence, we obtain
\begin{equation}
 \left\|u(t)-\exp(i\Delta t)v^{\pm}  \right\|_{l^p}\longrightarrow 0\,\,\,{\rm as}\,\,\,
 t\rightarrow  \pm \infty.\nonumber
\end{equation}
and the proof is complete.

\  \  $\Box$

\vspace*{0.5cm}

With the  asymptotic behavior of the solutions we  prove also
 Weinstein's conjecture  for the general power DNLS (\ref{eq:system2}) with $\sigma \ge 2/d$ \cite{Weinstein}: If $\nu(u_0)=||u_0||_{l^2}^2<\nu_{exc}$, then for any $p\in (2,\infty]$ the solutions decay, that is,
 \begin{equation}
  \left\| u(t)\right\|_{l^p}\rightarrow 0,\qquad{\rm as}\,\,\,|t|\rightarrow \infty.\nonumber
 \end{equation}
 
In fact, for the solution
\begin{equation}
 u(t)=\exp(i\Delta t)u_0-i \gamma\int_0^t\,\exp(i\Delta (t- s))|u(s)|^{2\sigma}u(s)ds, \nonumber
\end{equation}
we derive in an analogous manner to above,
\begin{eqnarray}
 ||u(t)||_{l^p}&\le& C ||u_0||_{l^{p^\prime}}\frac{1}{t^{d(p-2)/(3p)}}+\gamma \left\|\int_0^t\,\exp(i\Delta (t- s)) |u(s)|^{2\sigma}u(s)ds\right\|_{l^p}\nonumber\\
 &\le& C_1||u_0||_{l^{p^\prime}}\frac{1}{t^{d(p-2)/(3p)}}.\nonumber
\end{eqnarray}
Hence, 
\begin{equation}
  \left\| u(t)\right\|_{l^p}\rightarrow 0,\qquad{\rm as}\,\,\,t\rightarrow \infty.
  \nonumber
 \end{equation}
Decay for $t\rightarrow -\infty$ is verified analogously.
 
 In \cite{Stefanov} the authors presented a proof that for {\it sufficiently small solutions,
 their decay is like that of the  free solutions in the corresponding $l^p$ norms
 and this statement implies Weinstein's conjecture}. However, 
 our current result in Theorem \ref{theorem:scatterin1} is valid for 
 solutions with any $l^2$ norm below the threshold $\nu_{exc}$,
 as it is actually stated in Weinstein's conjecture, and 
 thus not restricted to 'sufficiently small solutions' as in \cite{Stefanov}. 
 Take, for example, $d=1$. Then $\nu_{tresh}=(2\sigma+2)/\gamma)^{1/\sigma}$. As noted  in the Introduction, the system  (\ref{eq:system1}) can be rendered independent of the  nonlinearity parameter $\gamma$
 and hence, we can set it equal to one
 in $\nu_{exc}$ giving $\nu_{exc}=(2(\sigma+1))^{1/\sigma}$
 which is not necessarily of small magnitude.

\vspace*{0.5cm}

Remains the treatment of the case $V_0\neq 0$. The operator $-\Delta_\delta=-\Delta-V_\delta$ possesses an absolutely continuous spectrum
$[0,4d]$. If $V_0>0$ ($V_0<0$) there is a
single eigenvalue (eigenenergy) (cf. (\ref{eq:eigenvalue0})) $\lambda<0$ ($\lambda>4d$) whose
eigenfunction represents the single linear non-staggering (staggering) bound state.

With concern to the absolutely continuous spectrum of $\Delta_\delta$, governing the scattering features, we have the following:
\begin{proposition}
The absolutely continuous spectrum of $\Delta_\sigma$ coincides with those of the unperturbed discrete Laplacian $\Delta$, i.e.
\begin{equation}
 \sigma_{cont}(\Delta_\delta)=\sigma_{cont}(\Delta).\nonumber
\end{equation}
\end{proposition}

\begin{proof}
Consider the multiplication operator $V_\delta\,:\,l^p \mapsto l^p$,
  $V=(V_n)_{n\in \mathbb{Z}}$, $\sup_{n\in \mathbb{Z}} |V_n|<\infty$, $V_\delta x=(...,V_{-1} x_{-1}, V_0 x_0, V_1 x_1,...)$ for all $x=(...,x_{-1},x_0,x_1,...)\in l^p$, $1\le p \le \infty$.
 Since  $V_n\neq 0$ if and only if $n = 0$, the operator $V_\delta$ is a rank-one operator and hence, compact.

Then  by Weyl's Theorem, the essential spectrum of $\Delta$ coincides with the essential spectrum of $\Delta_\delta$. (Since $\Delta$ is a bounded and self-adjoint operator its residual spectrum is empty.) From $\sigma_{ess}(\Delta_\delta)= \sigma_{ess}(\Delta)$ and $\sigma_{cont}(\Delta)\subset \sigma_{ess}(\Delta)$  follows $\sigma_{cont}(\Delta_\delta)$ equals $\sigma_{cont}(\Delta)$
finishing the proof.
 \end{proof}

\vspace*{0.5cm}

Regarding the proof of Theorem \ref{theorem:scatterin2} we facilitate the dispersive estimates for discrete nonlinear Schr\"odinger equations with potentials
\begin{equation}
 i\dot{\phi}_n(t)=H\phi_n(t)=\left(-\Delta+V_n\right)\phi_n(t).\nonumber
\end{equation}
In particular, we utilize Theorem 4 in \cite{Stefanov2} stating that for fixed $\sigma>5/2$ and a generic potential $V\in l^1_{2\sigma-1}$ there exists a constant $C$ depending on $V$ so that
\begin{equation}
 \left\| \exp(itH)P_{ac}(H)\right\|_{l^1\rightarrow l^\infty}\le Ct^{-3/2},\label{eq:l1linfty}
\end{equation}
for any $t>0$.

The potential $V_n=V_0\delta_{n,0}$ is generic in the sense of Definition 1 in \cite{Stefanov2}.  Interpolation between
$\left\|\psi(t)\right\|_{l^\infty}\le Ct^{-3/2}\left\| \psi(0)\right\|_{l^1}$
and $\left\| \psi(t)\right\|_{l^2}=\left\| \psi(0)\right\|_{l^2}$ yields
\begin{equation}
 \left\| \exp(itH)P_{ac}(H)\psi(0)\right\|_{l^p}\le  Ct^{-3/2} \left\| \psi(0)\right\|_{l^{p^\prime}},\,\,\,\frac{1}{p}+\frac{1}{p^\prime}=1,
\end{equation}
from which follows
that for any $p\ge 2$, $ \left\| \exp(itH)P_{ac}(H)\psi(0)\right\|_{l^p}\le  t^{-(p-2)/(3p)}$.

Since the rate of decay for the free solutions and the one with a $\delta-$potential coincide, the proof of Theorem \ref{theorem:scatterin2} proceeds along the lines of the proof of Theorem \ref{theorem:scatterin1} and, thus, is omitted.

\vspace*{0.5cm}

 \begin{remark}
 While for $V_0<0$, $\sigma< \min\{1,2/d\}$ respectively $V_0=0$, $\sigma < 2/d$,  the solutions do not scatter,
 for $V_0<0$, $\sigma\ge \max\{1,2/d\}$ respectively $V_0=0$, $\sigma \ge 2/d$,  we have the {\bf dichotomy} result:

 \noindent (1) If $||u_0||_{l^2}<\nu_{exc}$, then the solutions scatter in both time directions.

 \noindent (2) If $||u_0||_{l^2}\ge \nu_{exc}$, the ground state is a global  non-scattering  solution.

 \end{remark}

\noindent{\bf Acknowledgment}

\noindent The author is grateful to a referee for careful reading and invaluable suggestions.

 \vspace*{0.5cm}


\end{document}